\documentclass[12]{article}
\usepackage{geometry}
\usepackage[utf8]{inputenc}
\usepackage{amsmath}
\usepackage{amssymb}
\usepackage{amsthm}
\usepackage{xcolor}
\usepackage{tikz}
\usepackage{blkarray}
\usepackage{cleveref}
\usepackage[sort,numbers]{natbib}
\usepackage{caption}
\usepackage{subcaption}
\usepackage{float}
\usepackage{comment}

\newtheorem{defn0}{Definition}[section]
\newtheorem{prop0}[defn0]{Proposition}
\newtheorem{thm0}[defn0]{Theorem}
\newtheorem{lemma0}[defn0]{Lemma}
\newtheorem{corollary0}[defn0]{Corollary}
\newtheorem{example0}[defn0]{Example}
\newtheorem{remark0}[defn0]{Remark}
\newtheorem{conjecture0}[defn0]{Conjecture}

\newenvironment{definition}{\medskip \begin{defn0}}{\end{defn0}}
\newenvironment{proposition}{\medskip \begin{prop0}}{\end{prop0}}
\newenvironment{theorem}{\medskip \begin{thm0}}{\end{thm0}}
\newenvironment{lemma}{\medskip \begin{lemma0}}{\end{lemma0}}
\newenvironment{corollary}{\medskip \begin{corollary0}}{\end{corollary0}}
\newenvironment{example}{\medskip \begin{example0}\rm}{\end{example0}}
\newenvironment{remark}{ \medskip\begin{remark0}\rm}{\end{remark0}}

\newcommand{\tp}{\operatorname{top}}
\newcommand{\dg}{\operatorname{deg}}
\newcommand{\CC}{\mathbb{C}}
\newcommand{\EE}{\mathbb{E}}
\newcommand{\RR}{\mathbb{R}}
\newcommand{\pa}{\text{pa}}

\numberwithin{equation}{section}

\usepackage{authblk}
\title{Third-Order Moment Varieties of Linear \\ Non-Gaussian Graphical Models}
\author[1]{Carlos Am\'endola}
\author[1]{Mathias Drton}
\author[1]{Alexandros Grosdos}
\author[1]{\\Roser Homs}
\author[2]{Elina Robeva}
\affil[1]{Department of Mathematics, Technical University of Munich, Germany}
\affil[2]{Department of Mathematics, University of British Columbia, Canada}
\date{}

\begin{document}

\maketitle

\begin{abstract}
In this paper we study linear non-Gaussian graphical models from the perspective of algebraic statistics. These are acyclic causal models in which each variable is a linear combination of its direct causes and independent noise. The underlying directed causal graph can be identified uniquely via the set of second and third order moments of all random vectors that lie in the corresponding model. Our focus is on finding the algebraic relations among these moments for a given graph.
We show that when the graph is a polytree these relations form a toric ideal.  
We construct explicit {\em trek-matrices} associated to 2-treks and 3-treks in the graph. Their entries are covariances and third order moments and their $2$-minors define our model set-theoretically.  Furthermore, we prove that their 2-minors also generate the vanishing ideal of the model. Finally, we describe the polytopes of third order moments and the ideals for models with hidden variables.
\end{abstract}
    
\section{Introduction}

Featuring prominently in a wide variety of applications, directed graphical models~\cite{handbook} capture intuitive cause-effect relations among a set of random variables by hypothesizing that each variable is a noisy function of its causes.  For a number of statistical tasks, such as model selection, it has proven useful to obtain insights about the algebraic structure of the moments of the joint distributions in the graphical model for a given graph~\cite{drton:2018,handbook:chap3,seth:book}.  A prominent example are results on algebraic relations among second moments, i.e., covariances, in models that postulate linear functional relationships among the variables \cite{drton:robeva:weihs:2020,sullivant2008,VanOmmenMooij_UAI_17}.  Available results include, in particular, a characterization of the vanishing of subdeterminants of the covariance matrix \cite{MR2662356,MR3044565}, which covers conditional independence in Gaussian models as a special case.

In contrast to earlier work in algebraic statistics, we conduct here a first systematic study on algebraic relations that also involve higher moments of a linear directed graphical model. Specifically, we focus on second and third order moments.  Higher moments are crucial for non-Gaussian models.  Indeed, in non-Gaussian cases, different models may have identical structure in second but not in all higher moments, allowing for the development of methods for causal discovery \cite{shimizu:lingam:2006,shimizu:directlingam:2011,handbook:chap:spirtes}.
In this area, an algebraic perspective features, for instance, in \cite{wang:drton:2020}, where a particular type of relations among second and third moments was used to design a model selection method suitable for high-dimensional data.  

We begin by reviewing linear non-Gaussian graphical models. Let $G=(V,E)$ be a directed acyclic graph (DAG), and let $(X_i$, $i\in V)$, be a collection of random variables that represent statistical observations indexed by the vertices in $V$.  A vertex $j\in V$ is a parent of vertex $i$ if there is an edge pointing from $j$ to $i$, i.e., if $(j,i)\in E$ which we will also write as $j\to i\in E$.  We denote the set of all \emph{parents} of $i$ by $\pa(i)$. The graph $G$ gives rise to the {\em linear structural equation model} consisting of the joint distributions of all random vectors $X = (X_i, i\in V)$ such that
\begin{equation}
\label{eq:sem}
    X_i=\sum_{j\in \pa(i)}\lambda_{ji}X_j+\varepsilon_i, \quad i\in V,
\end{equation}
where the $\varepsilon_i$ are mutually independent random variables representing stochastic errors. The errors are assumed to have expectation $\EE[\varepsilon_i]=0$, finite variance $\omega^{(2)}_i=\EE[\varepsilon_i^2]>0$, and finite third moment $\omega^{(3)}_i=\EE[\varepsilon_i^3]$.  No other assumption about their distribution is made and, in particular, the errors need not be Gaussian (in which case we would have $\EE[\varepsilon_i^3]=0$ by symmetry of the Gaussian distribution).  The coefficients $\lambda_{ji}$ in \eqref{eq:sem} are unknown real-valued parameters, and we fill them into a matrix $\Lambda=(\lambda_{ji})\in\mathbb{R}^{V\times V}$ by adding a zero entry when $(j,i)\notin E$. We denote the set of all such sparse matrices as $\RR^E$.  We note that for simplicity, and without loss of generality, the equations in \eqref{eq:sem} do not include a constant term, so we have $\EE[X_i]=0$ for all $i\in V$.

\begin{example}
\label{ex:12}
  The simplest non-trivial case pertains to two random variables $X_1$ and $X_2$ and the DAG $G$ with vertex set $V=\{1,2\}$ and edge set $\{1\to 2\}$.  Then the model specified via \eqref{eq:sem} postulates that $X_1=\varepsilon_1$ and $X_2=\lambda_{12}X_1+\varepsilon_2$.  Let $S=(s_{ij})=(\EE[X_iX_j])$ be the covariance matrix of $(X_1,X_2)$.  Using that $\varepsilon_1$ and $\varepsilon_2$ are independent and have zero mean, we find
  \begin{align}
  \label{eq:ex:12}
      s_{11}\equiv\EE[X_1^2]&=\omega^{(2)}_1, &s_{12}\equiv\EE[X_1X_2]&=\lambda_{12}\omega^{(2)}_{1}, &s_{22}\equiv\EE[X_2^2]&=\omega^{(2)}_2+\lambda_{12}^2\omega^{(2)}_{1}.
  \end{align}
  As we assume $\omega^{(2)}_1,\omega^{(2)}_2>0$, the covariance matrix $S$ given by \eqref{eq:ex:12} is positive definite.  Conversely, setting $\omega^{(2)}_1=s_{11}>0$, $\lambda_{12}=s_{12}/s_{11}$, $\omega^{(2)}_2=s_{22}-s_{12}^2/s_{11}>0$ shows that any positive definite matrix may arise as covariance matrix of $(X_1,X_2)$.
  By symmetry, the set of covariance matrices associated to the DAG $G'$ that contains the single edge $1\leftarrow 2$ is again the entire positive definite cone.  The two DAGs $G$ and $G'$, thus, induce identical second moment structure for $(X_1,X_2)$; this is a (trivial) instance of the general phenomenon of Markov equivalence of DAGs \cite{handbook}.
  
  Turning to third moments, for the graph $G$ with edge $1\to 2$  we have that
  \begin{align}
  \label{eq:ex:12-third}
      t_{111}&\equiv \EE[X_1^3]=\omega^{(3)}_1, &t_{112}&\equiv\EE[X_1^2X_2]=\lambda_{12}\omega^{(3)}_{1},\\ t_{122}&\equiv\EE[X_1X_2^2]=\lambda_{12}^2\omega^{(3)}_{1},
      &t_{222}&\equiv\EE[X_2^3]=\omega^{(3)}_2+\lambda_{12}^3\omega^{(3)}_{1}.
  \end{align}
  The third moments now exhibit an algebraic relation, namely,
  \begin{equation}
      \label{eq:ex:12-trelation}
  t_{111} t_{122}-t_{112}^2 =0.
  \end{equation}
  The ideal of all algebraic relations among the covariances and third moments is generated by 
  \begin{equation}
  \label{eq:ex:12-strelations}
  \{  s_{12} t_{111} - s_{11} t_{112},\;s_{12} t_{112} - s_{11} t_{122},\;t_{111} t_{122}-t_{112}^2\},
  \end{equation}
  or equivalently, by the $2$-minors of the matrix
\begin{equation}
    \label{ex12:Aij}
    A_{1,2} \;=\;
    \begin{pmatrix}
    s_{11}&t_{111}&t_{112}\\
    s_{12}&t_{112}&t_{122}
    \end{pmatrix}.
\end{equation}  
None of the polynomials in \eqref{eq:ex:12-strelations} vanishes over the set of moments for the model given by the second graph $G'$ (for which indices $1$ and $2$ would be switched).  Hence, the third moment structure {\em differs} for $G$ versus $G'$ and provides a way to distinguish the two hypotheses encoded by $G$ and $G'$.  
  
The first relation in \eqref{eq:ex:12-strelations} was used in \cite[p.~45]{wang:drton:2020}.  It arises from the fact that  $s_{12}/s_{11}=t_{112}/t_{111}=\lambda_{12}$ for graph $G$.  The ratio based on second moments corresponds to the usual least squares estimator.  For asymmetrically distributed $\varepsilon_1$, the ratio of third moments provides an alternative (less efficient) estimator.
\end{example}

Returning to our general setting and forming the random vectors $X=(X_i)_{i\in V}$ and $\varepsilon=(\varepsilon_i)_{i\in V}$, the model from \eqref{eq:sem} states that 
\begin{equation}
\label{eq:sem-vector}
    X \;=\; \Lambda^T X+\varepsilon.
\end{equation}
When $\Lambda\in\RR^E$, with $E$ being the edge set of a DAG, it holds that $\det(I-\Lambda)=1$ and \eqref{eq:sem-vector} always admits the unique solution
\begin{equation}
\label{eq:sem-vector-solved}
    X \;=\;(I-\Lambda)^{-T}\varepsilon.
\end{equation}
Let $\Omega^{(2)}=(\EE[\varepsilon_i\varepsilon_j])$ and $\Omega^{(3)}=(\EE[\varepsilon_i\varepsilon_j\varepsilon_k])$ be the covariance matrix and the tensor of third moments of $\varepsilon$, respectively.  Both $\Omega^{(2)}$ and $\Omega^{(3)}$ are diagonal, with the diagonal entries being $\Omega^{(2)}_{ii}=\omega^{(2)}_i=\EE[\varepsilon_i^2]>0$ and $\Omega^{(3)}_{iii}=\omega^{(3)}_i=\EE[\varepsilon_i^3]$.

\begin{proposition}
  The covariance matrix and the third moment tensor of the solution $X$ of \eqref{eq:sem-vector} are equal to 
  \begin{align*}
      S &=(s_{ij})= (I-\Lambda)^{-T}\Omega^{(2)}(I-\Lambda)^{-1}, \\
      T &=(t_{ijk}) =\Omega^{(3)}\bullet (I-\Lambda)^{-1}\bullet (I-\Lambda)^{-1}\bullet(I-\Lambda)^{-1},
  \end{align*}
  respectively. Here $\bullet$ denotes the Tucker product.  
\end{proposition}

This fact follows from standard results about how moments of random vectors change after linear transformation; we include a proof to recall this calculus and the definition of the Tucker product.
\begin{proof}
    With $X_i=\sum_{k\in V} [(I-\Lambda)^{-T}]_{ik}\varepsilon_k$, the covariances are
    \begin{align*}
        s_{ij}=\EE[X_iX_j] &= \sum_{a,b\in V} [(I-\Lambda)^{-T}]_{ia}[(I-\Lambda)^{-T}]_{jb}\EE[\varepsilon_a\varepsilon_b]&= \sum_{a,b\in V} [(I-\Lambda)^{-T}]_{ia}[(I-\Lambda)^{-T}]_{jb}\Omega^{(2)}_{ab}
    \end{align*}
    as claimed by the matrix product formula.  Similarly, the third moments are
    \begin{align*}
        t_{ijk}=\EE[X_iX_jX_k] &= \sum_{a,b,c\in V} [(I-\Lambda)^{-T}]_{ia}[(I-\Lambda)^{-T}]_{jb}[(I-\Lambda)^{-T}]_{kc}\EE[\varepsilon_a\varepsilon_b\varepsilon_c] \\
        &=\sum_{a,b,c\in V} [(I-\Lambda)^{-T}]_{ia}[(I-\Lambda)^{-T}]_{jb}[(I-\Lambda)^{-T}]_{kc}\Omega^{(3)}_{abc},
    \end{align*}
    which by definition is the $(a,b,c)$-entry in the given Tucker product.    
\end{proof}

We emphasize that as we are assuming positive error variances, $\EE[\varepsilon_i^2]>0$, the matrix $\Omega^{(2)}$ is positive definite and the same is true for the covariance matrix $S$ of $X$.  Since $\Omega^{(3)}$ is diagonal, the third moment tensor $T$ of $X$ is a symmetric tensor of symmetric tensor rank at most $|V|$; this need not be the case for a general $V\times V\times V$ tensor \cite{comon:2008}.  In the sequel, we write $\mathit{PD}(V)$ for the positive definite cone in $\RR^{V\times V}$ and $\mathrm{Sym}_3(V)$ for the space of symmetric tensors in $\RR^{V \times V \times V}$.

\begin{definition}
  Let $G = (V,E)$ be a DAG.
  The \emph{third-order moment model} of $G$ is the set $\mathcal{M}^{\leq 3}(G)$ that comprises all pairs of covariance matrices and third moment tensors that are realizable under the linear structural equation model given by $G$. That is,
\begin{align*}
\mathcal M^{\leq 3}(G) = \big\{ &\big( (I-\Lambda)^{-T}\Omega^{(2)}(I-\Lambda)^{-1},\; \Omega^{(3)}\bullet (I-\Lambda)^{-1}\bullet (I-\Lambda)^{-1}\bullet(I-\Lambda)^{-1}\big) : \\
&\Omega^{(2)}\in\mathit{PD}(V) \text{ diagonal},\; \Omega^{(3)}\in \mathrm{Sym}_3(V) \text{ diagonal}, \; \Lambda\in\mathbb R^E\big\} \;\subseteq\;\mathit{PD}(V)\times \mathrm{Sym}_3(V).
\end{align*}
Furthermore, the \emph{third-order moment ideal} of $G$ is the ideal $\mathcal I^{\leq 3}(G)$ of polynomials in the entries $S=(s_{ij})$ and $T=(t_{ijk})$ that vanish when  $(S,T)\in\mathcal M^{\leq 3}(G)$.
\end{definition}

In this paper, we are focusing on models $\mathcal M^{\leq 3}(G)$ in the case where $G$ is a polytree, i.e., a directed acyclic graph whose underlying undirected graph is a tree.  The rest of the paper is organized as follows.
In Section~\ref{sec:simple-trek} we show how to parametrize the third-order moment model using the simple trek rule (\Cref{th:param}), a variation of the trek rule that reveals the toric structure of $\mathcal I^{\leq 3}(G)$ in the case of polytrees (\Cref{cor:toricTrees}). Sections from~\ref{sec:trek-matrices} to Section~\ref{sec:polytopes} mainly focus on polytrees and exploit their algebraic-geometric properties.
 Section~\ref{sec:trek-matrices} introduces trek-matrices (\Cref{def:matrices}) and provides algebraic relations that are sufficient to characterize when a pair of second and third order moments belongs to the model $\mathcal M^{\leq 3}(G)$ (\Cref{th:subset}).  In Section~\ref{sec:idealgens}, we obtain a Markov basis of the vanishing ideal of $\mathcal M^{\leq 3}(G)$ (\Cref{th:generators}).   Section~\ref{sec:latent} focuses on DAGs with latent variables. For polytrees, we can provide generators of the ideal of observed variables in terms of the trek-matrices (\Cref{prop:latent}).
In Section~\ref{sec:polytopes} we present the third-order moment polytope associated to a polytree and highlight the differences with respect to the covariance polytope in \cite{sullivant2008}. Finally, Section~\ref{sec:nontrees} displays several computational examples of non-trees.  

\section{Simple Trek Parametrization}
\label{sec:simple-trek}

The graphs that appear in this paper are acyclic, 
i.e., they do not contain a closed directed path. 
We may thus assume that the vertices of the graph are topologically ordered,
so that if $i \to j \in E$, then $i<j$. We now define the concept of \emph{multitrek} introduced in \cite{robeva2020}.

\begin{definition}
\label{def:trek}
A \emph{$k$-trek} between $k$ vertices $i_1,\dots,i_k$ is an ordered collection of directed paths $\tau = (P_1,\dots,P_k)$ where $P_r$ has sink $i_r$ and they all share the same source $v$. 
We call $v$ the \emph{top} of the trek and denote it by $\tp{\tau}$.  A $k$-trek is called \emph{simple} if its top node is the only node lying on all the $k$ directed paths that form the trek.
We denote the set of simple $k$-treks between  $i_1,\dots,i_k$ by  $\mathcal{T}(i_1,\dots,i_k)$.
\end{definition}

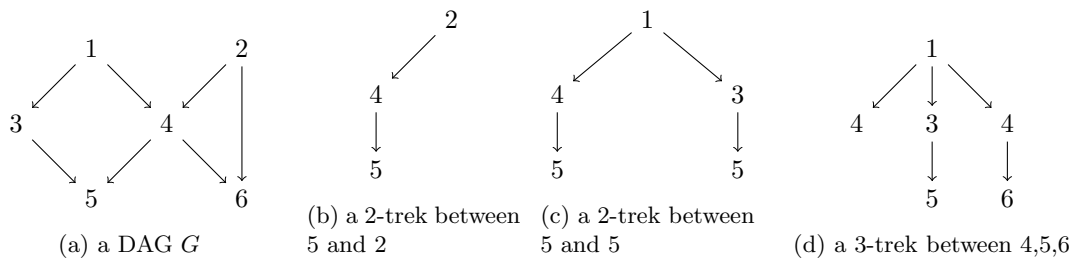
\begin{figure}[h]
\centering
\begin{subfigure}[b]{0.28\textwidth}
         \centering
\begin{tikzpicture}
\node(1) at (1,2) []    {$1$};
\node(2) at (3,2) []      {$2$};
\node(3) at (0,1) []      {$3$};
\node(4) at (2,1) []      {$4$};
\node(5) at (1,0) []      {$5$};
\node(6) at (3,0) []      {$6$};
\draw
(1) edge[->] (3)
(1) edge[->] (4) 
(2) edge[->] (4) 
(2) edge[->] (6) 
(3) edge[->] (5) 
(4) edge[->] (5) 
(4) edge[->] (6) ;
\end{tikzpicture}
         \caption{a DAG $G$}
         \label{fig:trekdag}
     \end{subfigure}
     \hfill
          \begin{subfigure}[b]{0.19\textwidth}
         \centering
         \begin{tikzpicture}
\node(2) at (1,2) []    {$2$};
\node(4) at (0,1) []      {$4$};
\node(5) at (0,0) []      {$5$};
\draw
(1) edge[->] (4)
(4) edge[->] (5);
\end{tikzpicture}
         \caption{a $2$-trek between 5 and 2}
         \label{fig:trek3}
     \end{subfigure}
     \hfill
 \begin{subfigure}[b]{0.19\textwidth}
         \centering
         \begin{tikzpicture}
\node(1) at (1.2,2) []    {$1$};
\node(4) at (0,1) []      {$4$};
\node(5) at (0,0) []      {$5$};
\node(3) at (2.4,1) []      {$3$};
\node(5b) at (2.4,0) []      {$5$};
\draw
(1) edge[->] (4)
(1) edge[->] (3) 
(4) edge[->] (5)
(3) edge[->] (5b);
\end{tikzpicture}
         \caption{a $2$-trek between 5 and 5}
         \label{fig:trek2}
     \end{subfigure}
     \hfill
     \begin{subfigure}[b]{0.28\textwidth}
         \centering
         \begin{tikzpicture}
\node(1) at (1,2) []    {$1$};
\node(4a) at (0,1) []      {$4$};
\node(3) at (1,1) []      {$3$};
\node(4b) at (2,1) []      {$4$};
\node(5) at (1,0) []      {$5$};
\node(6) at (2,0) []      {$6$};
\draw
(1) edge[->] (4a)
(1) edge[->] (3) 
(1) edge[->] (4b)
(3) edge[->] (5) 
(4b) edge[->] (6);
\end{tikzpicture}
         \caption{a $3$-trek between 4,5,6}
         \label{fig:trek4}
     \end{subfigure}
     \hfill
\caption{A DAG and some of its treks.}
\label{fig:trek}
\end{figure}
\begin{example}
In Figure \ref{fig:trek} we depict a graph $G$ and three of its treks.
The paths of the treks can use the same edge of the graph several times, as seen in Fig.~\ref{fig:trek4}.
Note also that the directed paths forming the trek might also be empty,
as is the case for the right path in Fig.~\ref{fig:trek3}.
In fact, the only simple $k$-trek between the $k$ vertices $i, \dots, i$ is the trek consisting of empty paths,
making the trek in Fig.~\ref{fig:trek2} nonsimple.
\end{example}

The statistical significance of 2-treks is that they provide a combinatorial way to parametrize covariance matrices through what is known as the trek rule \cite[Section 2]{MR2662356}.
This extends analogously to third-order moment tensors via $3$-treks.

Furthermore, an advantage of working with DAGs is that one can simplify 
the description given by the trek rule by focusing only on simple treks.
In \cite[Section 2]{sullivant2008} it is shown that simple 2-treks 
give a different parametrization of the covariance matrix,
by introducing a new set of indeterminates $a_i$ for each vertex in the graph defined by
\begin{equation}\label{eq:transfai}
    a_i:=\sum_{l_1\in pa(i)}\sum_{l_2\in pa(i)}\lambda_{l_1,i}\lambda_{l_2,i}s_{l_1,l_2}+\omega_i^{(2)}.
\end{equation}
This triangular,
and thus invertible,
transformation allows one to recursively pass from the diagonal entries in the matrix $\Omega^{(2)}$ to the $a_i$.
The same is true for the entries of the tensor $\Omega^{(3)}$, 
via the transformation (\ref{eq:transfbi}) in the proof of Proposition \ref{th:param} below.
We therefore obtain the shorter \emph{simple trek rule parametrization}
induced by the polynomial ring map
$$\begin{array}{rrcl}
\phi_G: & \CC[s_{ij},t_{ijk}\mid 1\leq i\leq j\leq k\leq n] & \rightarrow & \CC[a_i,b_i,\lambda_{ij}\mid i\rightarrow j\in E],\\
& s_{ij} & \mapsto & \displaystyle\sum_{\tau \in \mathcal{T}(i,j)} a_{\tp(T)}\displaystyle\prod_{k\rightarrow l\in \tau}\lambda_{kl},\\
& t_{ijk} & \mapsto & \displaystyle\sum_{\tau \in \mathcal{T}(i,j,k)}b_{\tp(T)}\displaystyle\prod_{m\rightarrow l\in \tau}\lambda_{ml}.
\end{array}$$
Dually, we get the rational parametrization
\begin{equation*}
\begin{array}{rrcl}
\phi_G^\ast: & \RR^{V}\times\RR^{V}\times\RR^{E} & \rightarrow & \RR^{\binom{n+1}{2}}\times\RR^{\binom{n+2}{3}}\\
& (a,b,\lambda) & \mapsto & \left(\left(\displaystyle\sum_{\tau\in \mathcal{T}(i,j)} a_{\tp(T)}\displaystyle\prod_{k\rightarrow l\in \tau}\lambda_{kl}\right)_{ij},  \left(\displaystyle\sum_{\tau\in \mathcal{T}(i,j,k)}b_{\tp(\tau)}\displaystyle\prod_{m\rightarrow l\in \tau}\lambda_{ml}\right)_{ijk}\right).
\end{array}
\end{equation*}

\begin{proposition}[Simple trek rule]\label{th:param} Let $G=(V,E)$ be a DAG, and let $\phi_G$ the ring morphism defined above. Then
the map $\phi_G^\ast$ gives a parametrization of the model $\mathcal{M}^{\leq 3}(G)$, and, therefore, $\mathcal I^{\leq 3}(G)=\ker \phi_G$.
\end{proposition}

\begin{proof}
For $ i\leq j<k$, the third moments are
$$t_{ijk}:=\EE(X_iX_jX_k)=\sum_{l\in pa(k)}\lambda_{lk}t_{ijl},$$
for $i<j$ we have \[t_{ijj}:=\EE(X_iX_j^2)=\sum_{l_1\in pa(j)}\sum_{l_2\in pa(j)}\lambda_{l_1j}\lambda_{l_2j}t_{il_1l_2},\]
and for all $i$ we have
$$t_{iii}:=\EE(X_i^3)=\sum_{l_1\in pa(i)}\sum_{l_2\in pa(i)}\sum_{l_3\in pa(i)}\lambda_{l_1,i}\lambda_{l_2,i}\lambda_{l_3,i}t_{l_1l_2l_3}+\omega_i^{(3)}.$$
Set
\begin{equation}\label{eq:transfbi}
    b_i:=\sum_{l_1\in pa(i)}\sum_{l_2\in pa(i)}\sum_{l_3\in pa(i)}\lambda_{l_1,i}\lambda_{l_2,i}\lambda_{l_3,i}t_{l_1,l_2,l_3}+\omega_i^{(3)}.
\end{equation}
An analogous construction for the entries of the covariance matrix, $s_{ij}$ via (\ref{eq:transfai})
is displayed in \cite{sullivant2008}. 
In Proposition 2.3 of that paper, Sullivant proves the simple trek rule parametrization for covariance matrices by checking that 
the covariance matrices in the image of $\phi_G^\ast$ for valid parameters $a$ are precisely the covariance matrices in $\mathcal{M}^{\leq 3}(G)$.

Next we prove that the set of third moment tensors in $\mathcal{M}^{\leq 3}(G)$ is the second component of $\phi_G^\ast(\RR^{V}\times\RR^{V}\times\RR^{E})$.
By definition, $t_{iii}=b_i$ for all $i$. Assume that 
$$t_{ijk}=\left(\displaystyle\sum_{\tau \in \mathcal{T}(i,j,k)}b_{\tp(\tau)}\displaystyle\prod_{m\rightarrow l\in \tau}\lambda_{ml}\right)_{ijk}$$ for all $1\leq i\leq j\leq k<r$. Then,
\begin{equation*}
t_{ijr}  = \sum_{l\in pa(r)}\lambda_{lr}t_{ijl}  = \sum_{l\in pa(r)}\lambda_{lr}\sum_{\tau \in \mathcal{T}(i,j,l)}b_{\tp(\tau)}\prod_{m\rightarrow q\in \tau}\lambda_{mq}  = \sum_{\tau \in \mathcal{T}(i,j,r)}b_{\tp(\tau)}\prod_{m\rightarrow q\in \tau}\lambda_{mq}.
\end{equation*}
The last equality comes from the fact that any $3$-trek in $\mathcal{T}(i,j,r)$ is a union of a $3$-trek in $\mathcal{T}(i,j,l)$ and an edge $l\rightarrow r$ for any parent $l$ of $r$. 
Similarly,
\begin{equation*}
    t_{irr} = \sum_{l\in\pa(r)}\lambda_{lr}t_{ilr} = \sum_{l\in\pa(r)}\lambda_{lr}\sum_{\tau\in\mathcal T(i,l,r)}b_{\tp(\tau)}\prod_{m\to q\in\tau}\lambda_{mq} = \sum_{\tau\in\mathcal T(i,r,r)}b_{\tp(\tau)}\prod_{m\to q\in \tau}\lambda_{mq},
\end{equation*}
which completes the induction.
\end{proof}

\begin{example} Consider the DAG $G$ in \Cref{fig:triangle}. 
The corresponding map $\phi_G$ is
$$\begin{array}{rl}
s_{ii} & \mapsto a_i \\
t_{iii} & \mapsto b_i \\
s_{12} & \mapsto a_1\lambda_{12}\\
s_{13} & \mapsto a_1(\lambda_{13}+\lambda_{12}\lambda_{23})\\
s_{23} & \mapsto a_1\lambda_{13}\lambda_{12}+a_2\lambda_{23}\\
t_{123} & \mapsto b_1\lambda_{12}(\lambda_{13}+\lambda_{12}\lambda_{23})\\
\vdots
\end{array}$$
The ideal $\mathcal{I}^{\leq 3}(G)$ is generated by 14 quadratic binomials and 14 cubics either in variables $t$ or mixed variables, such as $t_{123}^2-t_{122}t_{133}$, $s_{13}t_{123}-s_{12}t_{133}$,
$t_{123}t_{133}t_{222}-2t_{122}t_{133}t_{223}+t_{113}t_{223}^2+t_{122}t_{123}t_{233}-t_{113}t_{222}t_{233}$, $s_{23}t_{113}t_{222}-s_{23}t_{112}t_{223}-s_{22}t_{113}t_{223}+s_{11}t_{223}^2+s_{22}t_{112}t_{233}-s_{11}t_{222}t_{233}.$
\end{example}

\begin{figure}[h]
\centering
\begin{subfigure}[b]{0.45\textwidth}
         \centering
\begin{tikzpicture}
\node(v1) at (0,0) []      {$1$};
\node(v2) at (1,1) []      {$2$};
\node(v3) at (2,0) []      {$3$};
\draw(v1) edge[->] (v2)   (v1) edge[->] (v3) (v2) edge[->] (v3);
\end{tikzpicture}
         \caption{a DAG that is not a polytree}
         \label{fig:triangle}
     \end{subfigure}
     \hfill
          \begin{subfigure}[b]{0.45\textwidth}
         \centering
         \begin{tikzpicture}
\node(1) at (0,1) []      {$1$};
\node(2) at (-1.5,0) []      {$2$};
\node(3) at (-0.5,0) []      {$3$};
\node(4) at (0.5,0) []      {$4$};
\node(5) at (1.5,0) []      {$5$};
\draw(1) edge[->] (2)   (1) edge[->] (3) (1) edge[->] (4) (1) edge[->] (5);
\end{tikzpicture}
         \caption{a polytree}
         \label{fig:tetrapus}
     \end{subfigure}
\caption{Two DAGs}
\label{fig:dags}
\end{figure}
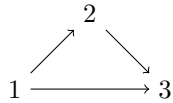
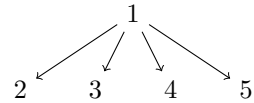 

In case the graph $G$ is a \emph{polytree}, the simple trek parametrization simplifies to a monomial map as $\mathcal{T}(i_1,\dots,i_k)$ has at most one element: the unique simple $k$-trek between $i_1,\ldots, i_k$, if it exists.  Therefore, we have the following fact.

\begin{corollary}\label{cor:toricTrees}
If $G$ is a polytree, then the ideal $\mathcal{I}^{\leq 3}(G)$ is toric.
\end{corollary}

\begin{example}
Let $G$ be the graph in \Cref{fig:tetrapus}.
The corresponding ideal $\mathcal{I}^{\leq 3}(G)$ is toric and therefore generated by binomials.  We will revisit this example in the next section, where we give a 
concrete generating set
given as $2$-minors of suitable matrices.
\end{example}

\section{Low-Rank Trek-Matrices}\label{sec:trek-matrices}

In this section we find equations defining the model $\mathcal M^{\leq 3}(G)$ when $G$ is a polytree. 
In this case, the model has a monomial parametrization, and we show that it is cut out by quadratic binomials.
These arise from 2-minors of matrices $A_{i,j}$ whose columns encode certain 2-treks and 3-treks involving $i$ and $j$ as endpoints.
In fact these matrices are crucial not only for finding polynomials that cut out the variety,
but also for the description of generating sets of the ideal of the model, as will be shown in \Cref{sec:idealgens}.

Since the simple $k$-trek between $i_1,\dots,i_k$ on a polytree, if it exists, is unique, then the notion of top is well-defined just by setting the end vertices.

\begin{definition} Let $G$ be a polytree and let $i_1,\dots,i_k$ be vertices such that a $k$-trek exists. We define $\tp(i_1,\dots,i_k)$ as the top node of its unique simple $k$-trek.
\end{definition}

\begin{definition}\label{def:matrices}
Let $G$ be a polytree. Let $i,j\in V$ be two vertices such that a 2-trek between $i$ and $j$ exists. We define the \emph{trek-matrix} between $i$ and $j$ as
    $$A_{i,j} := \begin{pmatrix} s_{ik_1} & \cdots & s_{ik_r} & t_{i \ell_1 m_1} & \cdots & t_{i \ell_q m_q} \\
    s_{jk_1} & \cdots & s_{jk_r} & t_{j \ell_1 m_1} & \cdots & t_{j \ell_q m_q}
    \end{pmatrix},$$
    where 
    \begin{itemize}
    \item $k_1,\ldots, k_r$ are vertices such that 
    $\tp(i,k_a)=\tp(j,k_a)$ for $a=1,\dots,r$, and
    \item $(l_1,m_1)$,$\dots$,$(l_q,m_q)$ are such that $\tp(i,l_b,m_b)=\tp(j,l_b,m_b)$ for $b=1,\dots,q$.
    \end{itemize}
\end{definition}

\begin{example}\label{ex:tetrapus1} Consider the graph $G$ in Figure~\ref{fig:tetrapus}. 
The corresponding trek-matrix $A_{1,2}$ from \Cref{def:matrices} is
\begin{equation*} 
\begin{blockarray}{cccccccccccccccc}
& 1	& 3 & 4 & 5 & 11 & 12 & 13 & 14 & 15 & 23 & \cdots &35 & 44 & 45 & 55 \\[2pt]
\begin{block}{c(ccccccccccccccc)}
1 & s_{11} & s_{13} & s_{14} & s_{15} & t_{111} & t_{112} & t_{113} &  t_{114} & t_{115} & t_{123} &   \cdots &  t_{135} & t_{144} & t_{145} & t_{155} \\
2 & s_{12} & s_{23} & s_{24} & s_{25} & t_{112} & t_{122} & t_{123} & t_{124} & t_{125} & t_{223} &  \cdots & t_{235} & t_{244} & t_{245} & t_{255} \\
\end{block}
\end{blockarray},
\end{equation*}
where the only missing column indices are $2$ and $22$ because $\tp(1,2)$ differs from $\tp(2,2)$ and $\tp(1,2,2)$ differs from $\tp(2,2,2)$. 
\end{example}

\begin{lemma}\label{lem:subgens}
Let $G=(V,E)$ be a polytree, and $i,j \in V$. Then
\begin{enumerate}
\item[(a)] if there is no 2-trek between $i$ and $j$, then $s_{ij} \in \mathcal I^{\leq 3}(G)$, and for all $k\in V$ such that there is no 3-trek between $i,j$ and $k$, we have that $ t_{ijk}\in\mathcal I^{\leq 3}(G)$.
\item[(b)] if there is an edge between $i$ and $j$, the $2$-minors of the trek-matrix $A_{i,j}$ lie in $\mathcal I^{\leq 3}(G)$.
\end{enumerate}
\end{lemma}

\begin{proof} 
(a) For all $i,j\in V$ such that there is no trek between them, by the trek rule, $\phi_G(s_{ij})=0$. Similarly, by the 3-trek-rule, if there is no 3-trek between $i,j,k$, then $\phi_G(t_{ijk})=0$. 

\smallskip
(b) Since $\tp(i,k_a)=\tp(j,k_a)$ and $\tp(i,l_b,m_b)=\tp(j,k_b,m_b)$, then $\phi_G(s_{ik_a})=\lambda_{ij}\phi_G(s_{jk_a})$ and $\phi_G(t_{il_bm_b})=\lambda_{ij}\phi_G(t_{jl_bm_b})$. Thus, the second row of $A_{i,j}$ is the first row of $A_{i,j}$ multiplied by $\lambda_{ij}$ and all 2-minors of $A_{i,j}$ belong to $\ker\phi_G=\mathcal I^{\leq 3}(G)$.
\end{proof}

\begin{theorem}\label{th:subset} Let $G$ be a polytree and $J$ be the ideal generated by the linear generators of $\mathcal{I}^{\leq 3}(G)$ and the $2$-minors of the matrices $A_{i,j}$ 
for $i \rightarrow j \in E.$
Then,
\[\mathcal{M}^{\leq 3}(G) =  \mathcal{V}(J) \cap (PD(V) \times \mathrm{Sym}_3(V)).\]
\end{theorem}
\begin{proof}
Let $S$ and $T$ be the second and third moments of a distribution such that $(S, T)\in\mathcal{M}^{\leq 3}(G)$. For any edge $i\to j \in E$, consider the matrix $A_{i,j}$. Note that it contains as a column the vector $\begin{bmatrix}s_{ii}, s_{ij}
\end{bmatrix}^T$. Let $\lambda_{ij} = \frac{s_{ij}}{s_{ii}}$, and let $\Lambda \in\mathbb R^E$ contain these as its entries.
Recall from \Cref{lem:subgens} that the ratio between every other column vector in $A_{i,j}$ will also be $\lambda_{ij}$. 
Consider
$$S' = (I-\Lambda)^T S (I-\Lambda), \quad T' = T\bullet (I-\Lambda)\bullet(I-\Lambda)\bullet (I-\Lambda).$$
We will now show that $S'$ and $T'$ are diagonal. For any $i\neq j$, we have that
$$s'_{ij} = \sum_{k,\ell} (I-\Lambda)_{ki}s_{k\ell}(I-\Lambda)_{\ell j}$$
$$= s_{ij} - \sum_{k\to i\in E} \lambda_{ki}s_{kj} - \sum_{\ell\to j\in E} \lambda_{lj} s_{i\ell} + \sum_{k\to i, \ell\to j\in E}\lambda_{ki}\lambda_{\ell j} s_{k\ell}.$$
We are going to consider 3 cases.

\underline{Case 1:} There is no 2-trek between $i$ and $j$. In this case, $s_{ij}=0$. Furthermore, if a 2-trek between $k$ and $j$ exists, then we can complete it to a 2-trek between $i$ and $j$. Thus, $s_{kj}=0$, and the first sum is $0$. Similarly, the other two sums are also 0. Therefore, $s'_{ij}=0$.

\underline{Case 2:} There is a 2-trek between $i$ and $j$. Set $v=\tp(i,j)$.

\underline{Case 2.0:} $v\neq i,j$. The simple 2-trek between $i$ and $j$ is of the form $i \leftarrow i_0 \cdots \leftarrow v\to  \cdots \to j_0 \to j$. Then, in the sums $\sum_{k\to i\in E} \lambda_{ki}s_{kj}$, 
$\sum_{\ell\to j\in E} \lambda_{lj} s_{i\ell}$, and $\sum_{k\to i, \ell\to j\in E}\lambda_{ki}\lambda_{\ell j} s_{k\ell}$ the only nonzero terms appear when $k=i_0$ and $\ell = j_0$ (otherwise, there cannot be a 2-trek between $k$ and $l$). When $k=i_0$, since $\tp(i,j) = \tp(k,j)$, $(s_{ij}, s_{kj})^T$ is a column in the matrix $A_{k,i}$, therefore  $\lambda_{ki} = \frac{s_{ij}}{s_{kj}}$. Similarly, when $\ell=j_0$, $(s_{jk}, s_{\ell k})^T$ is a column in the matrix $A_{\ell, j}$ and hence $\lambda_{\ell j} = \frac{s_{kj}}{s_{k\ell}}$. Therefore,
$$s'_{ij} = s_{ij} - \frac{s_{ij}}{s_{kj}} s_{kj} - \frac{s_{ij}}{s_{i\ell}}s_{i\ell} + \frac{s_{ij}}{s_{kj}}\frac{s_{kj}}{s_{k\ell}}s_{k\ell} = 0.$$

Note that either $s_{kj}\neq 0$ (resp. $s_{il}\neq 0$) or both $s_{kj}=s_{ij}=0$ (resp. $s_{il}=s_{ij}=0$), hence $s'_{ij}$ is either way well-defined. This situation is often repeated along the rest of the proof.

\underline{Case 2.1:} $v=i\neq j$. The simple 2-trek between $i$ and $j$ is of the form: $i\to\cdots\to j_0\to j$ or $i\leftarrow i_0 \cdots\leftarrow j$. WLOG, consider the former case. Then, the sums $\sum_{\ell\to j\in E} \lambda_{lj} s_{i\ell}$ and $\sum_{k\to i, \ell\to j\in E}\lambda_{ki}\lambda_{\ell j} s_{k\ell}$ are nonzero only when $\ell = j_0$ (otherwise, $s_{i\ell}$ and $s_{k\ell}$ are 0 because $j$ is a collider on the shortest path between $i$ and $\ell$ and the one between $j$ and $\ell$). 
Thus, we have that
$$s'_{ij}= s_{ij} - \sum_{k\in \text{pa}(i)} \lambda_{ki}s_{kj} -  \lambda_{j_0j} s_{ij_0} + \sum_{k\in\text{pa}(i)}\lambda_{ki}\lambda_{j_0 j} s_{kj_0}.$$
Note that $(s_{ji}, s_{j_0i})^T$ is a column in the matrix $A_{j_0,j}$, since top$(i,j_0) = $ top$(i,j) = i$. Thus,
$$s'_{ij}= s_{ij} - \frac{s_{ij}}{s_{ij_0}}s_{ij_0} + \sum_{k\in\text{pa}(i)}\lambda_{ki}(\lambda_{j_0j}s_{kj_0} - s_{kj}).$$

Furthermore, since $(s_{kj}, s_{kj_0})^T$ is a column in $A_{j_0,j}$, then
$$s'_{ij}=\sum_{k\in\text{pa}(i)}\lambda_{ki} (\frac{s_{kj}}{s_{kj_0}}s_{kj_0} - s_{kj}) = 0.$$
Therefore, $S'$ is diagonal.
\smallskip

We now proceed to showing that $T'$ is diagonal. We have that
$$t'_{ijk} = \sum_{a,b,c} (I-\Lambda)_{ai}(I-\Lambda)_{bj}(I-\Lambda)_{ck}t_{abc}$$
$$= t_{ijk} - \sum_{a\to i\in E}\lambda_{ai}t_{ajk} - \sum_{b\to j\in E}\lambda_{bj}t_{ibk} - \sum_{c\to k\in E} \lambda_{ck}t_{ijc}+ \sum_{a\to i, b\to j\in E}\lambda_{ai}\lambda_{bj}t_{abk}$$
$$+ \sum_{a\to i, c\to k\in E}\lambda_{ai}\lambda_{ck}t_{ajc} + \sum_{b\to j, c\to k\in E}\lambda_{bj}\lambda_{ck}t_{ibc} - \sum_{a\to i, b\to j, c\to k\in E}\lambda_{ai}\lambda_{bj}\lambda_{ck}t_{abc}$$

\underline{Case 1:} There is no 3-trek between $i,j,k$, i.e., $t_{ijk}=0$. Then, all of the 7 sums in the above expression will also be 0 by an analogous reasoning to Case 1 for $S'$.

\underline{Case 2:} There is a 3-trek between $i,j,k$. Set $v=\tp(i,j,k)$.

\underline{Case 2.0:} $v\neq i,j,k$. Therefore, the simple 3-trek between $i,j,k$ consists of 3 nontrivial directed paths $v\to\cdots\to i_0\to i, v\to\cdots \to j_0\to j$, and $v\to \cdots\to k_0\to k$. Because $G$ is a polytree, the seven sums above can be nonzero if and only if $a=i_0,b=j_0,c=k_0$ (otherwise, there cannot be a 3-trek between $a,b$, and $c$). Thus,
$$t'_{ijk} = t_{ijk} - \lambda_{i_0i}t_{i_0jk} - \lambda_{j_0j}t_{ij_0k} -  \lambda_{k_0k}t_{ijk_0}+ \lambda_{i_0i}\lambda_{j_0j}t_{i_0j_0k}$$
$$+ \lambda_{i_0i}\lambda_{k_0k}t_{i_0jk_0} + \lambda_{j_0j}\lambda_{k_0k}t_{ij_0k_0} - \lambda_{i_0i}\lambda_{j_0j}\lambda_{k_0k}t_{i_0j_0k_0}.$$

Since $\tp(i,j,k)=\tp(i_0,j,k)$, $(t_{ijk}, t_{i_0jk})^T$ is a column in $A_{i_0,i}$, and, therefore, $\lambda_{i_0i} = \frac{t_{ijk}} {t_{i_0jk}}$ (and similarly for the other $\lambda$'s in the expression above). Thus,
$$t'_{ijk}= t_{ijk} - \frac{t_{ijk}}{t_{i_0jk}}t_{i_0jk}- \frac{t_{ijk}}{t_{ij_0k}}t_{ij_0k}- \frac{t_{ijk}}{t_{ijk_0}}t_{ijk_0}+ \frac{t_{ij_0k}}{t_{i_0j_0k}}\frac{t_{ijk}}{t_{ij_0k}}t_{i_0j_0k}$$
$$+ \frac{t_{ijk_0}}{t_{i_0jk_0}}\frac{t_{ijk}}{t_{ijk_0}}t_{i_0jk_0} + \frac{t_{ijk_0}}{t_{ij_0k_0}}\frac{t_{ijk}}{t_{ijk_0}}t_{ij_0k_0} - \frac{t_{ijk}}{t_{i_0jk}}\frac{t_{i_0jk}}{t_{i_0j_0k}}\frac{t_{i_0j_0k}}{t_{i_0j_0k_0}}t_{i_0j_0k_0}\,=\,0.$$
\smallskip

\underline{Case 2.1:} $v=i\neq j,k$. In this case, the simple 3-trek between $i,j,k$ consists of the trivial path $i$ and paths $i\to\cdots\to j_0\to j$, and $i\to\cdots\to k_0\to k$. Therefore, we have that
$$t'_{ijk} = t_{ijk} - \sum_{a\in\text{pa}(i)}\lambda_{ai}t_{ajk} - \lambda_{j_0j}t_{ij_0k} -  \lambda_{k_0k}t_{ijk_0}+ \sum_{a\in\text{pa}(i)}\lambda_{ai}\lambda_{j_0j}t_{aj_0k}$$
$$+ \sum_{a\in\text{pa}(i)}\lambda_{ai}\lambda_{k_0k}t_{ajk_0} + \lambda_{j_0j}\lambda_{k_0k}t_{ij_0k_0} - \sum_{a\in\text{pa}(i)}\lambda_{ai}\lambda_{j_0j}\lambda_{k_0k}t_{aj_0k_0}$$
$$= t_{ijk} - \sum_{a\in\text{pa}(i)}\lambda_{ai}t_{ajk} - \frac{t_{ijk}}{t_{ij_0k}}t_{ij_0k} -  \frac{t_{ijk}}{t_{ijk_0}}t_{ijk_0}+ \sum_{a\in\text{pa}(i)}\lambda_{ai}\frac{t_{ajk}}{t_{aj_0k}}t_{aj_0k}$$
$$+ \sum_{a\in\text{pa}(i)}\lambda_{ai}\frac{t_{ajk}}{t_{ajk_0}}t_{ajk_0} + \frac{t_{ijk}}{t_{ij_0k}}\frac{t_{ij_0k}}{t_{ij_0k_0}}t_{ij_0k_0} - \sum_{a\in\text{pa}(i)}\lambda_{ai}\frac{t_{ajk}}{t_{aj_0k}}\frac{t_{aj_0k}}{t_{aj_0k_0}}t_{aj_0k_0} = 0.$$
\smallskip

\underline{Case 2.2:} $v=i=j\neq k$. In this case, the simple 3-trek between $i,j,k$ consists of the two trivial paths $i$ and $j$, and $i\to\cdots\to k_0\to k$. Therefore, we have that 
$$t'_{iik} = t_{iik} - \sum_{a\in\text{pa}(i)}\lambda_{ai}t_{aik} - \sum_{a\in\text{pa}(i)}\lambda_{ai}t_{iak} -  \lambda_{k_0k}t_{iik_0}+ \sum_{a,b\in\text{pa}(i)}\lambda_{ai}\lambda_{bi}t_{aak}$$
$$+ \sum_{a\in\text{pa}(i)}\lambda_{ai}\lambda_{k_0k}t_{aik_0} + \sum_{a\in\text{pa}(i)}\lambda_{ai}\lambda_{k_0k}t_{iak_0} - \sum_{a,b\in\text{pa}(i)}\lambda_{ai}\lambda_{bi}\lambda_{k_0k}t_{aak_0}$$
$$= t_{iik} - \sum_{a\in\text{pa}(i)}\lambda_{ai}t_{aik} - \sum_{a\in\text{pa}(i)}\lambda_{ai}t_{iak} -  \frac{t_{iik}}{t_{iik_0}}t_{iik_0}+ \sum_{a,b\in\text{pa}(i)}\lambda_{ai}\lambda_{bi}t_{aak}$$
$$+ \sum_{a\in\text{pa}(i)}\lambda_{ai}\frac{t_{aik}}{t_{aik_0}}t_{aik_0} + \sum_{a\in\text{pa}(i)}\lambda_{ai}\frac{t_{iak}}{t_{iak_0}}t_{iak_0} - \sum_{a,b\in\text{pa}(i)}\lambda_{ai}\lambda_{bi}\frac{t_{aak}}{t_{aak_0}}t_{aak_0} = 0.$$

Therefore, in all cases, $t'_{ijk}=0$ unless $i=j=k$.

We have shown that
$$S' = (I-\Lambda)^T S (I-\Lambda), \quad T' = T\bullet (I-\Lambda)\bullet(I-\Lambda)\bullet (I-\Lambda)$$
are diagonal, and $S'$ is positive definite since $S$ is, which means that
$$S = (I-\Lambda)^{-T} S' (I-\Lambda)^{-1}, \quad T = T'\bullet (I-\Lambda)^{-1}\bullet(I-\Lambda)^{-1}\bullet (I-\Lambda)^{-1}$$
lie in our model.
\end{proof}

\begin{example}\label{ex:tetrapus2} Consider the graph $G$ in Figure~\ref{fig:tetrapus}. It follows from \Cref{th:subset} that the model $\mathcal{M}^{\leq 3}(G)$ is cut out by all 2-minors of the trek-matrices $A_{1,2}$ (see \Cref{ex:tetrapus1}), $A_{1,3}$, $A_{1,4}$ and $A_{1,5}$. In particular, the ideal $J$, as defined in \Cref{th:subset}, is minimally generated by 431 quadratic binomials.
\end{example}

\section{Ideal Generators for Polytrees}\label{sec:idealgens}

In the previous section we provided a set-theoretic description of the model $\mathcal{M}^{\leq 3}(G)$ in terms of the trek-matrices corresponding to directed edges of the polytree $G$ and linear equations corresponding to the case where treks do not exist. However, the quadratic binomials considered in \Cref{th:subset} are in general not sufficient to generate the vanishing ideal of the model $\mathcal{I}^{\leq 3}(G)$. We illustrate this with an example.

\begin{example} Consider again the polytree $G$ in Figure~\ref{fig:tetrapus}. Note that there are quadratic binomials, such as $s_{25}s_{34}-s_{24}s_{35}$ or $t_{255}t_{345}-t_{245}t_{355}$, that vanish on the model but do not arise as 2-minors of the trek-matrices listed in \Cref{ex:tetrapus2}. One can check that $\mathcal{I}^{\leq 3}(G)$ is minimally generated by 557 quadratic binomials. A \texttt{Macaulay2} computation shows that the remaining 125 generators of the vanishing ideal of the model arise as 2-minors of trek-matrices that do not correspond to directed edges $i\rightarrow j$, namely $A_{2,3}$, $A_{2,4}$, $A_{2,5}$, $A_{3,4}$, $A_{3,5}$ and $A_{4,5}$. 
\end{example}

The rest of this section is devoted to prove that $\mathcal{I}^{\leq 3}(G)$ is indeed generated by the linear generators from \Cref{th:subset} and the quadratic binomials occurring as 2-minors of all trek-matrices associated to a polytree.

\begin{proposition}
Let $G=(V,E)$ be a polytree, and $i,j \in V$. If there is a 2-trek between $i$ and $j$, the $2$-minors of the trek-matrix $A_{i,j}$ lie in $\mathcal I^{\leq 3}(G)$.
\end{proposition}

\begin{proof} 
Consider a trek-matrix $A_{i,j}$. Fix any two columns of the matrix with indices $k$ (out of the $s$-columns) and $l,m$ (out of the $t$-columns). Set $a:=\tp(i,k)=\tp(j,k)$ and $b:=\tp(i,l,m)=\tp(j,l,m)$.

The 2-treks between $i$ and $j$ with tops $a$ and $b$, respectively, factor through the simple 2-trek between $i$ and $j$, which has $\tp(i,j)$ as top. 
Moreover, there is exactly one path between two nodes of a polytree, hence both the 2-treks (encoded in $s_{ik}$ and $s_{jk}$) and 3-treks (encoded in $t_{ilm}$ and $t_{jlm}$) go through $\tp(i,j)$ as shown in Figure \ref{fig:choke}.

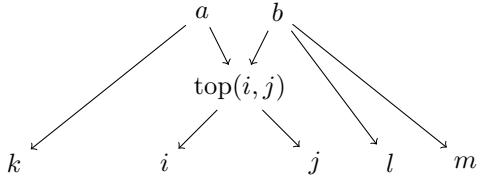
\begin{figure}[h]
\centering  
\begin{tikzpicture}
\node(a) at (0.5,2) []      {$a$};
\node(b) at (1.5,2) []      {$b$};
\node(t) at (1,1) []      {$\tp(i,j)$};
\node(k) at (-2,0) []      {$k$};
\node(i) at (0,0) []      {$i$};
\node(j) at (2,0) []      {$j$};
\node(l) at (3,0) []      {$l$};
\node(m) at (4,0) []      {$m$};
\draw(b) edge[->] (t)   (t) edge[->] (i) (t) edge[->] (j) (b) edge[->] (l) (b) edge[->] (m) (a) edge[->] (t) (a) edge[->] (k);
\end{tikzpicture}
    \caption{2-treks and 3-treks represented by $s_{ik},s_{jk},t_{ilm},t_{jlm}$ go through $\tp(i,j)$.}
    \label{fig:choke}
\end{figure}

From Figure \ref{fig:choke} and Proposition \ref{th:param}, it follows that $s_{ik}t_{jlm}-s_{jk}t_{ilm}$ belongs to $\ker\phi_G$.

Cases corresponding to two $s$ columns or two $t$ columns are dealt with analogously since only the central part of Figure \ref{fig:choke} matters. 
\end{proof}

\begin{remark}
Note that $\tp(i,j)$ is a \emph{choke point} between $I=\{i,j\}$ and $J=\{k,l,m\}$ in Figure \ref{fig:choke}, as defined in \cite{sullivant2008}. In terms of trek-separation \cite{MR2662356} this can be understood as the sets $I$ and $J$ being trek-separated by $(\{\tp(i,j)\},\{\emptyset\})$.

Even more, trek-matrices $A_{i,j}$ can be defined --- up to zero-columns that correspond to linear generators of the vanishing ideal --- in terms of trek-separation as follows:

    $$A_{i,j} := \begin{bmatrix} s_{ik_1} & \cdots & s_{ik_r} & t_{i \ell_1 m_1} & \cdots & t_{i \ell_q m_q} \\
    s_{jk_1} & \cdots & s_{jk_r} & t_{j \ell_1 m_1} & \cdots & t_{j \ell_q m_q}
    \end{bmatrix},$$
    where 
    \begin{itemize}
    \item all vertices $k_1,\ldots, k_r$  
    \item and all pairs of vertices $(l_1,m_1)$,$\dots$,$(l_q,m_q)$ 
    \end{itemize}
such that $(\{i,j\},\{k_1,\ldots, k_r,(l_1,m_1),\ldots,(l_q,m_q)\})$ are trek-separated by $(\{\tp(i,j)\},\{\emptyset\})$.
\end{remark}

For the covariance case, all quadratic binomials arise from 2-minors of $A_{i,j}$ as in \cite[Tetrad representation]{sullivant2008}. However, this is no longer the case for third moments or combinations of covariance and moments, as the next example shows.

\begin{example}\label{ex:quadratic_binomials}
For the graph in Figure \ref{fig:tetrapus}, $f=s_{23}t_{145}-s_{45}t_{123}\in \mathcal{I}^{\leq 3}(G)$, and because of the lack of repeated indices in $s$, this cannot come from a 2-minor of $A_{i,j}$.
However, this binomial is not necessary in order to generate the vanishing ideal of the model. Indeed,

$$f=(s_{23}t_{145}-s_{34}t_{125})+(s_{34}t_{125}-s_{45}t_{123})$$
is the sum of 2-minors of the trek-matrices $A_{2,4}$ and $A_{3,5}$, respectively.
\end{example}

In fact, all quadratic binomials in $\mathcal{I}^{\leq 3}(G)$ can be expressed as the sum of at most two 2-minors of trek matrices: 

\begin{proposition}
Let $G$ be a polytree. Then all quadratic binomials in $\mathcal{I}^{\leq 3}(G)$ are linear combinations of 2-minors of trek-matrices $A_{i,j}$ for $i,j \in V$.
\end{proposition}

\begin{proof}
Let $f=m_1+m_2\in I_G$ be a quadratic binomial. By \Cref{th:param}, the tops of the underlying 2- or 3-treks of $m_1$ must be pairwise equal to the tops associated to $m_2$. Moreover, the indices of the variables in $m_2$ must be a permutation of the indices of $m_1$.

\smallskip
Case I: $f$ is a quadratic binomial on variables $s$. Without loss of generality we can assume $f=s_{ij}s_{kl}-s_{il}s_{jk}$. If $\tp(i,j)=\tp(i,l)$ and $\tp(k,l)=\tp(k,j)$, then 

\begin{equation}\label{eq:ss}
\begin{array}{c|cc}
  & i       & k\\
\hline
j & s_{ij} & s_{kj}\\
l & s_{il} & s_{kl}
\end{array}
\end{equation}

\noindent
is a 2-minor of the matrix $A_{j,l}$ in Definition \ref{def:matrices}. 
Otherwise, since the tops of the underlying 2-treks must be pairwise equal, we have $\tp(i,j)=\tp(k,j)$ and $\tp(k,l)=\tp(i,l)$. Then $f$ is the determinant of a 2-minor of matrix $A_{i,k}$ (it is enough to interchange the elements in the skew-diagonal of (\ref{eq:ss}).

\smallskip
Case II: $f$ is a quadratic binomial on variables $t$. Then $f=t_{ijk}t_{lmn}-t_{\sigma(i)\sigma(j)\sigma(k)}t_{\sigma(l)\sigma(m)\sigma(n)}$, where the $\{i,j,k\}\cup\{\sigma(i),\sigma(j),\sigma(k)\}$ has either one or two elements.

If the intersection consists of a single element, then we can assume 
$f=t_{ijk}t_{lmn}-t_{ilm}t_{jkn}$. If $\tp(i,j,k)=\tp(j,k,n)$ and $\tp(i,l,m)=\tp(l,m,n)$, then

\begin{equation}\label{eq:tt1}
\begin{array}{c|cc}
  & j,k      & l,m\\
\hline
i & t_{ijk} & t_{ilm}\\
n & t_{jkn} & t_{lmn}
\end{array}
\end{equation}

\noindent is a 2-minor of $A_{i,n}$. Otherwise, we need to exchange the elements in the skew-diagonal in order to have equal tops $b_1:=\tp(i,j,k)=\tp(i,l,m)$ and $b_2:=\tp(j,k,n)=\tp(l,m,n)$ arranged in columns. The resulting matrix 

\begin{equation}\label{eq:tt2}
\begin{array}{c|cc}
  & i      & n\\
\hline
j,k & t_{ijk} & t_{jkn}\\
l,m & t_{ilm} & t_{lmn}
\end{array}
\end{equation}

\noindent has not the desired shape. Consider the two 4-treks between  $j,k,l$ and $m$, one with top $b_1$ and the other with top $b_2$. Because $G$ is a polytree, one must factor through the other. Assume $b_1\leq b_2$, then the determinant of (\ref{eq:tt2}) can be written as the sum of the determinants of the corresponding 2-minors in $A_{k,l}$ and $A_{j,m}$.

If the intersection consists of two elements, we can proceed analogously.

\smallskip
Case III: $f$ is a quadratic binomial in both variables $s$ and $t$. Then $f$ is of the form $s_{ij}t_{klm}-s_{\sigma(i)\sigma(j)}t_{\sigma(k)\sigma(l)\sigma(m)}$ where $\{i,j\}\cap\{\sigma(i),\sigma(j)\}$ is either a single element or empty. As in case II, it can be proved that either $f$ arises as a 2-minor of a matrix in \ref{def:matrices} or it is a sum of two such matrices.
\end{proof}

\begin{remark} In (\ref{eq:tt2}) of the proof above, we see that $t_{ijk}t_{lmn}-t_{ilm}t_{jkn}$ 
is the sum of the determinants of 

\begin{equation}\label{eq:2minors}
\begin{array}{c|cc}
A_{k,l}  & (i,j)      & (m,n)\\
\hline
k & t_{ijk} & t_{kmn}\\
l & t_{ijl} & t_{lmn}
\end{array}\quad\quad\mbox{and}\quad\quad
\begin{array}{c|cc}
A_{j,m}  & (i,l)      & (k,n)\\
\hline
j & t_{ijl} & t_{jkn}\\
m & t_{ilm} & t_{kmn}
\end{array}.
\end{equation}

\noindent
The indices in $A_{k,l}$ and $A_{j,m}$ 
indicate which paths in the 3-trek we exchange between the two terms of the determinant, as represented in Figure \ref{fig:factor2}.

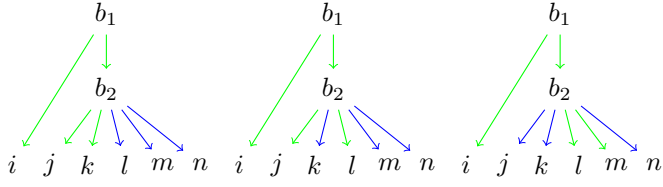
\begin{figure}[ht]
\centering  
\begin{tikzpicture}
\node(b1) at (0,2) []      {$b_1$};
\node(b2) at (0,1) []      {$b_2$};
\node(i) at (-1.25,0) []      {$i$};
\node(j) at (-0.75,0) []      {$j$};
\node(k) at (-0.25,0) []      {$k$};
\node(l) at (0.25,0) []      {$l$};
\node(m) at (0.75,0) []      {$m$};
\node(n) at (1.25,0) []      {$n$};
\draw[blue] (b2) edge[->] (l)   (b2) edge[->] (m) (b2) edge[->] (n);
\draw[green] (b1) edge[->] (i)   (b2) edge[->] (j) (b2) edge[->] (k) (b1) edge[->] (b2);
\end{tikzpicture}
\begin{tikzpicture}
\node(b1) at (0,2) []      {$b_1$};
\node(b2) at (0,1) []      {$b_2$};
\node(i) at (-1.25,0) []      {$i$};
\node(j) at (-0.75,0) []      {$j$};
\node(k) at (-0.25,0) []      {$k$};
\node(l) at (0.25,0) []      {$l$};
\node(m) at (0.75,0) []      {$m$};
\node(n) at (1.25,0) []      {$n$};
\draw[blue] (b2) edge[->] (k)   (b2) edge[->] (m) (b2) edge[->] (n);
\draw[green] (b1) edge[->] (i)   (b2) edge[->] (j) (b2) edge[->] (l) (b1) edge[->] (b2);
\end{tikzpicture}
\begin{tikzpicture}
\node(b1) at (0,2) []      {$b_1$};
\node(b2) at (0,1) []      {$b_2$};
\node(i) at (-1.25,0) []      {$i$};
\node(j) at (-0.75,0) []      {$j$};
\node(k) at (-0.25,0) []      {$k$};
\node(l) at (0.25,0) []      {$l$};
\node(m) at (0.75,0) []      {$m$};
\node(n) at (1.25,0) []      {$n$};
\draw[blue] (b2) edge[->] (k)   (b2) edge[->] (j) (b2) edge[->] (n);
\draw[green] (b1) edge[->] (i)   (b2) edge[->] (m) (b2) edge[->] (l) (b1) edge[->] (b2);
\end{tikzpicture}

    \caption{From left to right: pairs of 3-treks representing terms $t_{ijk}t_{lmn}$, $t_{ijl}t_{kmn}$, $t_{ilm}t_{jkn}$. First two pictures correspond to the determinant of the 2-minor of $A_{kl}$ displayed in (\ref{eq:2minors}) and last two correspond to $A_{jm}$.}
    \label{fig:factor2}
\end{figure}

\end{remark}

\begin{definition}\label{def:lex}
We consider the following lexicographic ordering on the variables:
\begin{itemize}
    \item $s_{ij}<t_{klm}$, or equivalently, 2-treks come before 3-treks;
    \item $s_{ij}< s_{kl}$ (analogously for 3-treks) iff either 
    \begin{itemize}
        \item $\tp(i,j)<\tp(k,l)$, i.e. treks with lower top values come first;
    \item or $\tp(i,j)=\tp(k,l)$ and the path from the top to $i$ goes through the node with the lowest value compared to the path from the top to $k$, or $i=k$ but the path from the top to $j$ goes through the node with the lowest value compared to the path from the top to $l$.
    \end{itemize}
\end{itemize}
\end{definition}

We can represent monomials in variables $s_{ij},t_{ijk}$  in a tableau with 3 columns. Each column encodes a path of the trek as a string of numbers denoting the nodes in the path. For 2-treks, we will leave the last column empty. Monomials can be represented by adding rows to the tableau.
Considering the previous lexicographic ordering, the tableau is uniquely determined. We call this a standard form tableau representation, see \cite{sullivant2008}.

\begin{example}\label{ex:tableau}
Consider the polytree $G$ in Figure \ref{fig:tetrapus} and take the cubic binomial $f=s_{24}s_{35}t_{344}-s_{34}s_{45}t_{234}$ in $\mathcal{I}^{\leq 3}(G)$. Using a standard form tableau representation this can be rewritten as
$$f=\left[\begin{array}{l|l|l}
  \underline{1}2  &  \underline{1}4 & \\
  \underline{1}3  &  \underline{1}5 & \\
  \underline{1}3  &  \underline{1}4 & \underline{1}4\\
\end{array}\right]
-\left[\begin{array}{l|l|l}
  \underline{1}3  &  \underline{1}4 & \\
  \underline{1}4  &  \underline{1}5 & \\
  \underline{1}2  &  \underline{1}3 & \underline{1}4\\
\end{array}\right].$$
\end{example}

In order to prove that all binomials in $\mathcal{I}^{\leq 3}(G)$ are generated by quadratic binomials we will extend Sullivant's strategy \cite{sullivant2008} of reducing disagreements via quadratic moves to both 2 and 3-treks using tableaux with 3 columns as in \Cref{ex:tableau}. 

\begin{example}\label{ex:reduce}
Let $f$ be as in \Cref{ex:tableau}. We will use the third row of the second term to reduce the disagreements between the two terms in the first row. The quadratic move
$$q=\left[\begin{array}{l|l|l}
  \underline{1}3  &  \underline{1}4 & \\
  \underline{1}2  &  \underline{1}3 & \underline{1}4\\
\end{array}\right]-\left[\begin{array}{l|l|l}
  \underline{1}2  &  \underline{1}4 & \\
  \underline{1}3  &  \underline{1}3 & \underline{1}4\\
\end{array}\right]=s_{34}t_{234}-s_{24}t_{334}\in \mathcal{I}^{\leq 3}(G)$$

yields a polynomial with fewer disagreements, namely 
$$f'=\left[\begin{array}{l|l|l}
  \underline{1}2  &  \underline{1}4 & \\
  \underline{1}3  &  \underline{1}5 & \\
  \underline{1}3  &  \underline{1}4 & \underline{1}4\\
\end{array}\right]
-\left[\begin{array}{l|l|l}
  \underline{1}2  &  \underline{1}4 & \\
  \underline{1}4  &  \underline{1}5 & \\
  \underline{1}3  &  \underline{1}3 & \underline{1}4\\
\end{array}\right].$$

The first row of $f'$  can be factored out and only remains a quadratic binomial already in $\mathcal{I}^{\leq 3}(G)$. In polynomial notation, we have $f'=f+s_{45} q=s_{24}(s_{35}t_{344}-s_{45}t_{334})\in \mathcal{I}^{\leq 3}(G)$, hence we can now write $f$ as a linear combination of quadratic binomials of the ideal.
\end{example}

Let $I_G$ be the ideal generated by matrices $A_{i,j}$ in Definition \ref{def:matrices} and linear generators $s_{ij}$ and $t_{ijk}$ when there is no 2-trek between $i,j$ or no 3-trek between $i,j,k$.

\begin{theorem}\label{th:generators}
Let $G$ be a polytree. Then all binomials in $\mathcal{I}^{\leq 3}(G)$ are generated by quadratic binomials, i.e., $$\mathcal{I}^{\leq 3}(G)=I_G.$$ 
\end{theorem}

\emph{Sketch of the proof:} For binomials in $\CC[s_{ij}]$, see \cite[Theorem 5.8]{sullivant2008}. In the remaining cases, we will reduce disagreements via quadratic moves on both 2 and 3-treks as in \Cref{ex:reduce}. 

For binomials in $\CC[t_{ijk}]$, we start by considering disagreements right after the top in the first row of the tableau representation of $f=m_1+m_2$. Thanks to the lexicographic ordering in \Cref{def:lex}, we can assume that disagreements are never in the first column. Moreover, the ordering is used to ensure the occurrence of the pattern of interest in the first row of $m_1$ in some inferior row of $m_2$. In the general case, the number of disagreements is reduced with a single quadratic move. 

The special case occurs when the quadratic move is not applicable because it does not correspond to variables in $\CC[t_{ijk}]$ and this is described by certain equalities among the entries in the tableau. If we can ensure the existence of a third row in $m_2$ with at most one occurrence of the pattern, we obtain two quadratic moves that reduce disagreements. There is a special case of this special case in which the two moves are not applicable. This is a very degenerate situation that turns out to be reducible in a single quadratic move.
On the other hand, if the existence is not ensured, it yields a certain equality and we deal with that particular case analogously.

Next we need to consider disagreements that appear not right after the top but later on. The strategy is the same but now the required quadratic moves might have different top than the 3-treks we are considering. If this occurs, we can think of them as partial moves.

Finally, we deal with binomials in $\CC[s_{ij},t_{ijk}]$ with both variables $s$ and $t$. Note that the fact that 2-treks always come before 3-treks (see \Cref{ex:reduce}) in the lexicographic ordering removes the assumption that disagreements are never in the first column but the rest of the proof follows analogously.

A complete proof of this theorem is presented in \Cref{app:proof}.

\medskip


\begin{example}
Recall the graph from Figure \ref{fig:tetrapus}. Let $J$ be the ideal in \Cref{th:subset} generated by all 2-minors of trek-matrices corresponding to edges of the graph, namely $A_{1,2}, A_{1,3}, A_{1,4},A_{1,5}$.
Computations in \texttt{Macaulay2} show that $\mathcal{I}^{\leq 3}(G)=(J:s_{11}^{\infty})$. In particular, it holds that $$\mathcal{M}^{\leq 3}(G)=\mathcal{V}(\mathcal{I}^{\leq 3}(G))\cap \left(\mathit{PD}(V)\times \mathrm{Sym}_3(V)\right)=\mathcal{V}(J)\cap \left(\mathit{PD}(V)\times \mathrm{Sym}_3(V)\right).$$
\end{example}

\section{Latent Variables} \label{sec:latent}

It often happens in statistical practice that some of the random variables cannot be observed.
Such variables are called latent or hidden.
These include psychological or sociological constructs such as 
anxiety or extroversion,
random variables coming from sensitive (e.g., medical) data that have been censored,
as well as data that have been lost, for instance an extinct biological species.
Latent variables imply the existence of correlations between observed variables, even if there are no causal relations among them.

In this section we show that for a class of graphs with hidden variables,
the observed variable ideal 
is generated by 2-minors of certain submatrices of the trek-matrices $A_{i,j}$ defined in Section \ref{sec:trek-matrices}.

\begin{definition}
Let $H \cup O$ be a partition of the nodes of the DAG $G$.
The hidden nodes $H$ are said to be \emph{upstream} from the observed nodes $O$ in G if there are no edges $o \rightarrow h$ in $G$ with $o \in O$ and $h \in H$. In this case, we call $H\cup O$ an {\em upstream partition} of the vertices.
\end{definition}

\begin{example}\label{ex:hidden} Recall the graph in Example \ref{ex:quadratic_binomials}. Consider the partition $G=H\cup O$ with $H=\{1\}$ and $O=\{2,3,4,5\}$. Let $\mathcal{I}_O^{\leq 3}(G)=\mathcal{I}^{\leq 3}(G)\cap\CC[s_{ij},t_{ijk}:i,j,k\in O]$ be the vanishing ideal of the model over the observed variables. 
Computations with \texttt{Macaulay2} show that
it is a binomial ideal irredundantly generated by 126 quadratic binomials.
In fact, this ideal is generated by minors of submatrices of the trek matrices $A_{i,j}$ for $i,j \in O$ after removing all columns containing variables with the first index, such as 
\begin{equation*} 
\begin{blockarray}{ccccccccccc}
& 4	& 5 & 23 & 24 & 25 & 34 & 35 & 44 & 45 & 55 \\[2pt]
\begin{block}{c(cccccccccc)}
2 & s_{24} & s_{25} & t_{223} & t_{224} & t_{225} & t_{234} & t_{235} &  t_{244} & t_{245} & t_{255}\\
3 & s_{34} & s_{35} & t_{233} & t_{234} & t_{235} & t_{334} & t_{335} & t_{344} & t_{345} & t_{355} \\
\end{block}
\end{blockarray}.
\end{equation*}
\end{example}

The behavior displayed in Example \ref{ex:hidden} is known to be true for vanishing ideals of partially observed Gaussian models arising from polytrees with an upstream partition\cite[Section 6]{sullivant2008}. 
In the rest of this section we focus on proving this result for ideals of third moments.

For an upstream partition $H\cup O$ we now define a multigrading on the ring $\CC[a,b,\lambda]$,
which induces a multigrading on $\CC[s,t]$ such that the corresponding moment ideal is homogeneous.
Let $\dg a_h = (1,1)$ for all $h \in H$ and $\dg a_o = (1,3)$ for all $o \in O$. Similarly let $\dg b_h = (1,0)$ for all $h \in H$ and $\dg b_o = (1,3)$ for all $o \in O$. Finally let $\dg \lambda_{ho} = (0,1)$ for all $h \in H$ and $o \in O$, and $\dg \lambda_{ij} = (0,0)$ otherwise.

\begin{lemma}\label{lem:grading}
For an upstream partition on a DAG $G$, 
the above grading induces a grading on $\CC[s,t]$ with
\[ \dg s_{ij} = (1, 1+\text{number of elements in the multiset } \{i,j\} \text{ in } O) \]
and
\[ \dg t_{ijk} = (1, \text{number of elements in the multiset } \{i,j,k\} \text{ in } O). \]
\end{lemma}

\begin{proof}
For any monomial $m$ that appears in the image of $s_{ij}$ there are two possibilities. 
That monomial arises from a $2$-trek whose top is either in the  observed part of the graph or in the hidden one.
By having an upstream partition, 
the first case can only occur if both $i$ and $j$ are observed vertices
and this also implies that 
all the $\lambda_{kl}$ arising from edges in the trek have degree $(0,0)$,
since all edges are between vertices in the observed part.
Therefore the degree of $m$ is equal to the degree of $a_{\tp(i,j)}$ which is $(1,3)$.
In the second case, the variable $a_{\tp(i,j)}$ will contribute the multidegree $(1,1)$. 
Again by having an upstream partition,
there will be exactly as many edges from a hidden vertex to an observed one as
number of elements in the multiset  $\{i,j\}$ in $O$.
But this implies
\[ \dg s_{ij} = (1, 1+\text{number of elements in the multiset } \{i,j\} \text{ in } O). \]
The argument transfers almost verbatim to the degree of $t_{ijk}$.
\end{proof}

\begin{proposition}\label{prop:latent}
For any polytree $G$ with an upstream partition of its edges, 
the \emph{observed variable ideal} given by the intersection \[\mathcal{I}^{\leq 3}_O(G) =\mathcal{I}^{\leq 3}(G)\cap\CC[s_{ij},t_{ijk}:i,j,k\in O] \]  is generated by the minors of the submatrices of $A_{i,j}$ with $i,j$ both in $O$,
with columns indexed by $k$ and pairs $(l,m)$ where $k, l, m$ are all in $O$.
\end{proposition}

\begin{proof}
We first show that for a polytree $G$,
the ideal $\mathcal{I}^{\leq 3}(G)$ is homogeneous with respect to the multigrading
introduced in Lemma \ref{lem:grading}.
By Theorem \ref{th:generators}, the ideal $\mathcal{I}^{\leq 3}(G)$ is generated by determinants of $2 \times 2$ matrices. 
Their rows are given by vertices $i$ and $j$.
Similarly the columns are parametrized either by vertices $k$ or pairs $(l,m)$, or a combination of the two.
In any case, one can readily check that the resulting binomial is homogeneous.

Since the ideal is homogeneous,
it is enough to show that each homogeneous polynomial $f \in \mathcal{I}^{\leq 3}_O(G)$ is generated by minors of the required form.
Since all indeterminates in $\CC[s_{ij},t_{ijk}:i,j,k\in O]$ have degree $(1,3)$,
we obtain $\dg f = (r,3r)$, for some natural number $r$.
Since $f$ belongs to $\mathcal{I}^{\leq 3}(G)$ it can be written as a finite sum
$f = \sum_{l} u_lM_l$,
where the $u_i$ are monomials and each $M_l$ is $2$-minor coming from a trek-matrix $A_{i,j}$. 
Since the minors $M_l$ are homogeneous,
we can decompose the sum into
\[f = \sum_{\dg u_lM_l = (r,3r)}u_lM_l + \sum_{\dg u_lM_l \neq (r,3r)}u_lM_l. \]
But then the second part must be zero, implying $f = \sum_{\dg u_lM_l = (r,3r)}u_lM_l$.
This can only be achieved if all indeterminates appearing in $u_l$ and $M_l$ are in the set $\{s_{ij},t_{ijk}:i,j,k\in O\}$.
Indeed, using any of the remaining indeterminates would force the degree to be $(p,q)$ with $q$ strictly less than $3p$.
Thus we have found a description of $f$ of the required form.
\end{proof}

\section{Moment Polytopes}\label{sec:polytopes}

A natural approach to understanding the algebro-geometric properties of a toric variety is to consider the polytope associated to its monomial parametrization. 
In this section we focus on the polytope associated to the third-order moments variety $\mathcal{V}(\mathcal{I}^3(G))$ following the steps of \cite{sullivant2008} in the covariance case. Note that $\mathcal{I}^3(G)=\mathcal{I}^{\leq 3}(G)\cap\CC[t_{ijk}]$ and this ideal is generated by the 2-minors of trek-matrices that correspond to columns in the variables $t_{ijk}$. 

Given a polytree $G = (V,E)$, for any minimal 3-trek between $i,j,k$, we define the vector $e_{ijk}\in\mathbb R^{|V| + |E|}$ of exponents of the monomial $\phi_G(t_{ijk})=b_{\tp(i,j,k)}\prod_{l\rightarrow m\in \mathcal{T}(i,j,k)}\lambda_{lm}\in\mathbb{R}[b_l,\lambda_{lm}]$. Let $(\mathbf{z},\mathbf{y})$ be the coordinates of $e_{ijk}$, where $\mathbf{z}=\left( z_l\right)_{l\in V}$ are the exponents of $b_\ell$ for $\ell\in V$ and $\mathbf{y}=\left(y_{lm}\right)_{l\rightarrow m\in E}$ are the exponents of $\lambda_{\ell m}$ for $\ell\to m\in E$.

\begin{definition}
Given a polytree $G$, its associated third-order moment polytope is 
$$P_G^{(3)}=\mathrm{conv}\left(e_{ijk}: i,j,k \mbox{ such that a 3-trek between $i$, $j$ and $k$ exists}\right).$$
\end{definition}

\begin{remark}\label{rem:coord}
Note that for any polytree the coordinates of $e_{ijk}$ take values in $\lbrace 0,1,2\rbrace$. Furthermore, the vector $\mathbf{z}$ has a single non-zero coordinate $z_l=1$ for $l=\tp(i,j,k)$. Moreover, in the minimal 3-trek between $i,j,k$ an edge can appear at most twice.
\end{remark}


\begin{theorem}\label{th:polytope}
For a fully observed polytree $G$,
the third-order moment polytope $P_G^{(3)}$ is the solution to the following set of equations and inequalities
\begin{equation}\label{eq1}
    z_l \geq 0 \text{ for all } l \in V, 
\end{equation}
\begin{equation}\label{eq2}
    y_{lm} \geq 0 \text{ for all } l \to m \in E, 
\end{equation}
\begin{equation}\label{eq3}
    \sum_{l \in V}z_l =1,
\end{equation}
\begin{equation}\label{eq4}
    2z_l + \sum_{h \in pa(l)}y_{hl} - y_{lm} \geq 0 \text{ for all m such that } l \to m \in E, 
\end{equation}
\begin{equation}\label{eq5}
    3z_l + \sum_{h \in pa(l)}y_{hl} - \sum_{m \in ch(l)}y_{lm} \geq 0 \text{ for all } l \in V.
\end{equation}
\end{theorem}

\begin{proof} Let $Q_G$ be the polytope defined by the inequalities above. 
\begin{enumerate}
    \item[Step 1.] $Q_G$ is bounded. Indeed, from (\ref{eq1}) and (\ref{eq3}), one gets $0\leq z_l\leq 1$. If $l$ is a source of the polytree, then $y_{lm}\leq 2z_l\leq 2$ by (\ref{eq4}). Moreover, (\ref{eq4}) gives a bound on any $y_{lm}$ in terms of $z_l$ and $y_{hl}$ for any $h\rightarrow l\in E$, which can be recursively bounded starting from a source node.
    \item[Step 2.] $P_G^{(3)}\subseteq Q_G$. It is enough to check that the defining inequalities of $Q_G$ are satisfied by every vertex $e_{ijk}=(\mathbf{z},\mathbf{y})$ in the polytope $P_G^{(3)}$. In other words, we need to check if the inequalities are satisfied by the exponents of variables $b_l$ and $\lambda_{l,m}$ in the minimal 3-trek $\tau$ between $i,j,k$.
    
    (\ref{eq1}), (\ref{eq2}) and (\ref{eq3}) are trivially satisfied by \Cref{rem:coord}. 
    
    To prove (\ref{eq4}) we start by fixing an edge $l\rightarrow m\in E$. Note that if $l\rightarrow m$ is not in $\tau$, then $y_{lm}=0$ and the inequation is trivially satisfied. Let us assume that $l\rightarrow m$ is in $\tau$. There are two possibilities:
    \begin{itemize}
        \item If $z_l=1$, then $l$ is the top of $\tau$, hence for any $h\rightarrow l$ we have $y_{hl}=0$. Moreover, the edge $l\rightarrow m$ can appear at most twice in $\tau$, namely $y_{lm}=1,\,2$:
        $$2z_l + \sum_{h \in pa(l)}y_{hl} - y_{lm}=2-y_{lm}\geq 0.$$
        \item If $z_l=0$, then the top of $\tau$ is an ancestor of $l$. There is a single edge $h\rightarrow l$ in $\tau$ (the one in the path from the top of $\tau$ to $l$) with $y_{hl}=1,\,2$. If $y_{lm}=2$, then $h\rightarrow l$ also appears twice, hence
        $$2z_l + \sum_{h \in pa(l)}y_{hl} - y_{lm}=y_{hl}-y_{lm}\geq 0.$$
    \end{itemize}
    Finally, a similar reasoning proves (\ref{eq5}):
     \begin{itemize}
        \item If $z_l=1$, $\tau$ has at most 3 outgoing edges from $l$ (counting multiplicity), hence
        $$3z_l + \sum_{h \in pa(l)}y_{hl} - \sum_{m \in ch(l)}y_{lm}=3-\sum_{m \in ch(l)}y_{lm}\geq 0.$$
        \item If $z_l=0$, $\tau$ has at most 2 ingoing edges to $l$ and the number of outgoing edges from $l$ cannot exceed the latter:
        $$3z_l + \sum_{h \in pa(l)}y_{hl} - \sum_{m \in ch(l)}y_{lm}=\sum_{h \in pa(l)}y_{hl}-\sum_{m \in ch(l)}y_{lm}\geq 0.$$
    \end{itemize}

    \item[Step 3.] $Q_G\subseteq P_G^{(3)}$. We prove that the extreme points of $Q_G$ are a subset of the extreme points of $P_G$. More precisely, we fix $(\mathbf{z}^0,\mathbf{y}^0)\in Q_G$ and want to prove that it can be expressed as
    \begin{equation}\label{eq:point2}
    (\mathbf{z}^0,\mathbf{y}^0)=\lambda(\mathbf{z}^1,\mathbf{y}^1)+(1-\lambda)(\mathbf{z}^2,\mathbf{y}^2)    
    \end{equation}
    for some $(\mathbf{z}^1,\mathbf{y}^1)=e_{ijk}$, $\lambda>0$ and $(\mathbf{z}^2,\mathbf{y}^2)\in Q_G$. Since $Q_G$ is bounded, the inclusion of extreme points follows.
    
    Without loss of generality we can assume that all $y^0_{lm}>0$. Otherwise the problem reduces to a smaller polytree or forest. The polytope associated to a forest is the direct join of the polytopes associated to its connected components. 
    This follows by the direct generalization of \cite[Proposition 3.6]{sullivant2008}  to third-moment ideals. Indeed, if $G$ is a disjoint union of two subgraphs $G_1$ and $G_2$, then
\begin{equation*}\label{prop:disjointgraphs}
    \mathcal{I}^{\leq 3}(G)=
    \mathcal{I}^{\leq 3}(G_1) 
    + \mathcal{I}^{\leq 3}(G_2)
  + \langle s_{ij}, t_{ikj} | i \in V(G_1), j \in V(G_2)  \rangle.
\end{equation*}
    
    We call an edge of $G$ where equality of (\ref{eq4}) holds for $(\mathbf{z}^0,\mathbf{y}^0)$ a \emph{marked edge}. 
    \begin{itemize}
        \item Claim 1: For each vertex in the graph, there is at most one outgoing marked edge.
   \end{itemize}
   
   It follows by Claim $1$, that we can build a minimal 3-trek $\tau$ between $i,j,k$ such that 
   \begin{enumerate}
       \item the coordinate of $\mathbf{z}^0$ corresponding to the vertex $\tp(i,j,k)$ is strictly positive, 
       \item $i,j,k$ are sinks of the graph $G$, 
       \item marked edges of $G$ are either contained in $\tau$ or not incident to $\tau$,
       \item if there is an outgoing marked edge from the top of $\tau$, it has multiplicity 2,
       \item the multiplicity of a marked edge is equal to the multiplicity of the previous edge in $\tau$.
   \end{enumerate}
  Let $e_{ijk}=(\mathbf{z}^1,\mathbf{y}^1)$ be the vector of exponents of $\tau$.
   \begin{itemize}
        \item Claim 2: The vector $(\mathbf{z}^2,\mathbf{y}^2)$ defined by \Cref{eq:point2} is in $Q_G$ for any $\lambda>0$ small enough.
    \end{itemize}
\end{enumerate}

\noindent
If both claims hold, then $Q_G=P_G^{(3)}$ and the theorem is proved.

\textit{Proof of Claim 1.} Let us assume that a vertex $l$ has two different outgoing marked edges $l\rightarrow m_1$ and $l\rightarrow m_2$, namely
\begin{equation}\label{eq:marked}
    2z_l^0 + \sum_{h \in pa(l)}y_{hl}^0 - y_{lm_i}^0=0,\quad{i=1,2}.
\end{equation}
Adding them up we obtain
$$4z_l^0 + 2\sum_{h \in pa(l)}y_{hl}^0-y_{lm_1}^0-y_{lm_2}^0=0,$$
and substracting $z_l^0 + \sum_{h \in pa(l)}y_{hl}^0$ yields
\begin{equation}\label{eq:contr}
 3z_l^0 + \sum_{h \in pa(l)}y_{hl}^0-y_{lm_1}^0-y_{lm_2}^0\leq 0.  
\end{equation}
If the inequality in (\ref{eq:contr}) was strict, then we would have a contradiction to (\ref{eq5}). Thus, (\ref{eq:contr}) must be an equality, hence
$z_l^0 + \sum_{h \in pa(l)}y_{hl}^0=0$. From (\ref{eq:marked}) it follows that $y_{lm_i}^0=0$ for $i=1,2$, which is a contradiction to the fact that all coordinates of $\mathbf{y}^0$ are strictly positive.

\textit{Proof of Claim 2.} We show that $(\mathbf{z}^2,\mathbf{y}^2)$ satisfies the defining equations of $Q_G$.
\begin{itemize}
    \item[(\ref{eq1})] For any $z^1_l=0$, $z^2_l\geq 0$ holds trivially. If $z^1_l=1$, then $l$ is the top of $\tau$, hence $z^0_l>0$ by construction of $\tau$ --- see condition (a). For $\lambda>0$ small, $z^0_l-\lambda=(1-\lambda)z^2_l\geq 0$.
    \item[(\ref{eq2})] Since $y_{lm}^0>0$ by assumption, an analogous argument follows for $\lambda>0$ small.
    \item[(\ref{eq3})] From $1=\sum_{l\in V} z_l^0=\lambda\sum_{l\in V} z_l^1+(1-\lambda)\sum_{l\in V} z_l^2=\lambda+(1-\lambda)\sum_{l\in V} z_l^2$, it follows that $\sum_{l\in V} z_l^2=1$.
    \item[(\ref{eq4})] For any $l\rightarrow m\in E$ we want to prove that 
    \begin{equation}\label{eq:proof4}
    2z_l^0 + \sum_{h \in pa(l)}y_{hl}^0 -y_{lm}^0-\lambda\left(2z_l^1+\sum_{h \in pa(l)}y_{hl}^1 - y_{lm}^1\right)\geq 0.    
    \end{equation}
    If $l\rightarrow m$ is not a marked edge, then
    $2z_l^0 + \sum_{h \in pa(l)}y_{hl}^0 -y_{lm}^0>0$ and hence (\ref{eq:proof4}) holds for $\lambda>0$ small enough.
    
    If $l\rightarrow m$ is a marked edge but $l\notin\tau$, then $2z_l^1+\sum_{h \in pa(l)}y_{hl}^1 - y_{lm}^1=0$ by construction of $\tau$ --- see condition (c) --- and hence (\ref{eq:proof4}) trivially holds.
    
    If $l\rightarrow m$ is a marked edge and $l\in\tau$, then (\ref{eq:proof4}) holds if and only if $2z_l^1+\sum_{h \in pa(l)}y_{hl}^1 - y_{lm}^1=0$. We need to study two cases separately:
    \begin{itemize}
      \item $z_l^1=1$: $2z_l^1+\sum_{h \in pa(l)}y_{hl}^1 - y_{lm}^1=2-y_{lm}$. (\ref{eq:proof4}) holds if and only if $y_{lm}=2$, , which is true by (d).
      \item $z_l^1=0$: $2z_l^1+\sum_{h \in pa(l)}y_{hl}^1 - y_{lm}^1=y_{hl}-y_{lm}$. (\ref{eq:proof4}) holds if and only if $y_{hl}=y_{lm}$, which is true by (e).
    \end{itemize}
    
    \item[(\ref{eq5})] We want to prove that 
    \begin{equation}\label{eq:proof5}
    3z_l^0 + \sum_{h \in pa(l)}y_{hl}^0 - \sum_{m \in ch(l)}y_{lm}^0-\lambda\left(3z_l^1+\sum_{h \in pa(l)}y_{hl}^1 - \sum_{m \in ch(l)}y_{lm}^1\right)\geq 0.    
    \end{equation}
    If $z_l^1=0$, then either $l$ is a sink of the polytree or $\sum_{h \in pa(l)}y_{hl}^1=\sum_{m \in ch(l)}y_{lm}^1=1,2$ --- $i,j,k$ are sinks of $\tau$ by (b). In the latter, $$3z_l^1+\sum_{h \in pa(l)}y_{hl}^1 - \sum_{m \in ch(l)}y_{lm}^1=\sum_{h \in pa(l)}y_{hl}^1 - \sum_{m \in ch(l)}y_{lm}^1=0.$$ 
    In the former, $\sum_{h \in pa(l)}y_{hl}^1$ cannot be cancelled out because $l$ has no children. Then (\ref{eq:proof5}) holds if and only if 
    $$3z_l^0 + \sum_{h \in pa(l)}y_{hl}^0 - \sum_{m \in ch(l)}y_{lm}^0>0.$$
    Since there are no children this translates into $3z_l^0 + \sum_{h \in pa(l)}y_{hl}^0$ being strictly positive, which is true by the assumption $y_{hl}^0>0$.
    
    If $z_l^1=1$, then either $l$ is a sink of the polytree or $\sum_{m \in ch(l)}y_{lm}^1=3$ (because $i,j,k$ are sinks). In the latter, since there are no parents of $l$ in $\tau$, then $3z_l^1+\sum_{h \in pa(l)}y_{hl}^1 - \sum_{m \in ch(l)}y_{lm}^1=0$. In the former, $3z_l^1$ cannot be cancelled out because there are no children and the same reasoning as in the case of sinks for $z_l^1=0$ applies. 
\end{itemize}
\end{proof}

\begin{example}\label{ex:polytope}
We revisit the graph $G$ from Figure \ref{fig:tetrapus}. Here, $P^{(3)}_G$ is a polytope in $\RR^9$ and its points have coordinates $(z_1,z_2,z_3,z_4,z_5,y_{12},y_{13},y_{14},y_{15})$. In particular, the boundary of the polytope contains points such as $e_{123}=(1,0,0,0,1,1,0,0)$ and $e_{223}=(1,0,0,0,2,1,0,0)$. By \Cref{th:polytope}, the defining inequalities of $P^{(3)}_G$ are, for $l\in\{1,2,3,4,5\}$ and $m\in\{2,3,4,5\}$, 

\begin{equation*}
z_l\geq 0,\quad y_{1m}\geq 0,\quad \sum_{l}z_l =1,
\quad 2z_1 - y_{1m} \geq 0,
\quad 3z_1 - \sum_{m}y_{1m} \geq 0.
\end{equation*}
\end{example}

In \cite[Section 5]{sullivant2008}, from the covariance polytope of a polytree $G$, Sullivant derives normality and Cohen-Macaulayness of $\CC[s_{ij}]/\mathcal{I}^2(G)$ via polyhedral techniques.
Unfortunately, this program cannot be directly extended to third-order moments. The main obstacle is that, as opposed to \cite[Corollary 5.2]{sullivant2008}, there exists no unimodular pulling triangulation of $P^{(3)}$ in the general case, as we detail in the next example.

\begin{example}\label{ex:nonsquarefree}
The existence of unimodular pulling triangulations of $P^{(3)}_G$ is equivalent to the existence of squarefree initial ideals of $\mathcal{I}^{3}(G)$ with respect to reverse lexicographic term orders; see \cite[Chapter 8]{sturmfels1996grobner} or \cite{haase2021existence} for more details. Continuing with the setting of \Cref{ex:polytope}, we will prove that for the considered star tree $G$ all reverse lexicographic term orders give rise to initial ideals with squares among its generators.

Note that $\mathcal{I}^{3}(G)$ is an ideal generated by quadratic binomials in $\CC[t_{ijk}]$, a ring with 35 variables that, by symmetry, can be classified into seven categories: $t_{111}$, $t_{iii}$, $t_{1ii}$, $t_{11i}$, $t_{iij}$, $t_{1ij}$ and $t_{ijk}$, where $i,j,k\in\lbrace 2,3,4,5\rbrace$ are all different. Reordering the variables within a category will not have an effect on the existence of squares, since such variables play the exact same role.

Since $t_{iij}t_{jkk}-t_{ijk}^2\in\mathcal{I}^{3}(G)$, one of the two monomials will belong to the minimal set of monomial generators of the initial ideal. If we want to avoid squares, then $t_{iij}>t_{ijk}$ must hold. In a similar fashion, considering the polynomials $t_{11i}t_{ijj}-t_{1ij}^2$ and 
$t_{iij}t_{jkk}-t_{ijk}^2$ in the ideal,
we obtain
\begin{equation}\label{eq:ordering}
t_{111},\,t_{1ii} > t_{11i} > t_{1ij},\quad
t_{iij} > t_{1ij},\,t_{ijk}.
\end{equation}

Finally, since $t_{iii}$ does not belong to the ideal, it is enough to computationally check all initial ideals with respect to reverse lexicographical orderings with all possible combinations of the remaining six categories of variables that satisfy (\ref{eq:ordering}). \texttt{Macaulay2} verifies that the initial ideals of all 28 possible arrangements of variables fail to be squarefree because of the unique cubic binomial of the form
$t_{iij}t_{ikl}^2-t_{iik}t_{iil}t_{jkl}$ in the corresponding Gröbner basis of $\mathcal{I}^3(G)$.
\end{example}

\section{The Unexplored Forest of Non-Trees}\label{sec:nontrees}
 
 In this final section we present some examples concerning graphs that are not trees and may not even be acyclic. Nevertheless, we observe that some features persist when considering their third-order moment varieties.
 
\begin{example}
\label{ex:7.1}
Let $G$ be the graph with edges 
$1 \rightarrow 2$,
$2 \rightarrow 3$ and 
$1 \rightarrow 3$
depicted in Figure \ref{fig:triangle}.
The ideal $\mathcal{I}^{\leq 3}(G)$ of the model is generated by the ideal $I_3$ of the $3$-minors of the matrix
\begin{equation*} 
\begin{blockarray}{cccccccc}
& 1 & 2 & 11 & 12 & 13 & 22 & 23 \\[2pt]
\begin{block}{c(ccccccc)}
1 & s_{11} & s_{12} & t_{111} & t_{112} & t_{113} & t_{122} & t_{123} \\
2 & s_{12} & s_{22} & t_{112} & t_{122} & t_{123} & t_{222} & t_{223} \\
3 & s_{13} & s_{23} & t_{113} & t_{123} & t_{133} & t_{223} & t_{223} \\
\end{block}
\end{blockarray},
\end{equation*}
plus the ideal $I_2$ of all $2$-minors of the submatrix
\begin{equation*} 
\begin{blockarray}{ccccc}
& 1 & 11 & 12 & 13 \\[2pt]
\begin{block}{c(cccc)}
1 & s_{11}  & t_{111} & t_{112} & t_{113}  \\
2 & s_{12}  & t_{112} & t_{122} & t_{123}  \\
3 & s_{13}  & t_{113} & t_{123} & t_{133} \\
\end{block}
\end{blockarray},
\end{equation*}
as well as the polynomial 
\begin{equation*}
    f = s_{23}t_{133}t_{222}-s_{23}t_{123}t_{223}-s_{22}t_{133}t_{223}+s_{13}t_{223}^3+s_{22}t_{123}t_{233}-s_{13}t_{222}t_{233}.
\end{equation*}
The polynomial $f$ is up to sign equal to the determinant of the matrix
\begin{equation*} 
\begin{blockarray}{cccc}
& 13 & 22 & 23 \\[2pt]
\begin{block}{c(ccc)}
\emptyset & s_{13} & s_{22} & s_{23}\\
2 & t_{123} & t_{222} & t_{223} \\
3 & t_{133} & t_{223} & t_{223} \\
\end{block}
\end{blockarray},
\end{equation*}
which is not a matrix of the form previously observed because of the row indexed by the empty set.
If one allows saturations, then $f$ becomes redundant
as we have
\[\mathcal{I}^{\leq 3}(G) = I_3+I_2 + \langle f \rangle = (I_3+I_2):s_{11}^{\infty}. \]
\end{example} 
\begin{example}
Let $G$ be the diamond graph on $4$ vertices with edges $0 \rightarrow 1$, $0 \rightarrow 2$, $1 \rightarrow 3$ and $2 \rightarrow 3$. The ideal $\mathcal{I}_O^{\leq 3}(G)$ is minimally generated by $77$ quadratic binomials ($2$-minors) and $138$ cubics.  It is possible to give a (non-minimal) set of generators that are $2$- or $3$-minors analogous to what we have seen for Example~\ref{ex:7.1}.

If we partition the random variables into hidden and observed ones as $H=\{0\}$ and $O=\{1,2,3\}$, then the ideal $\mathcal{I}_O^{\leq 3}(G)$ is given by the vanishing of $3$-minors of the matrix
\[
\begin{blockarray}{cccccccc}
& 1 & 2 & 11 & 12 & 13 & 22 & 23\\[2pt]
\begin{block}{c(ccccccc)}
1& s_{11} & s_{12} & t_{111} & t_{112} & t_{113} & t_{122} & t_{123}\\
2& s_{12} & s_{22} & t_{112} & t_{122} & t_{123} & t_{222} & t_{223}\\
3& s_{13} & s_{23} & t_{113} & t_{123} & t_{133} & t_{223} & t_{233}\\
\end{block}
\end{blockarray}.
\]
Note that for this observed variable ideal no indexing with the empty set is necessary.
\end{example}

It is also interesting to consider cyclic structures.

\begin{example}
Let $V = \{ 1, 2 \}$ and consider $G$ to be the $2$-cycle with edges $1 \rightarrow 2$ and $2 \rightarrow 1$. The model  $\mathcal M^{\leq 3}(G)$ is a hypersurface in $\mathit{PD}(V)\times \mathrm{Sym}_3(V)$. Indeed, the ideal $\mathcal{I}^{\leq 3}(G)$ is generated by the determinant of the matrix 
\begin{equation*} 
\begin{blockarray}{cccc}
& 11 & 12 & 22 \\[2pt]
\begin{block}{c(ccc)}
\emptyset & s_{11} & s_{12} & s_{22}\\
1 & t_{111} & t_{112} & t_{122} \\
2 & t_{112} & t_{122} & t_{222} \\
\end{block}
\end{blockarray},
\end{equation*}
so we note that even in the presence of cycles, there are matrices similar to the trek-matrices that define the model.
\end{example}

\begin{example}
Let $V = \{ 1, 2, 3 \}$ and consider $G$ to be the 3-cycle with edges $1 \rightarrow 2$, $2 \rightarrow 3$ and $3 \rightarrow 1$. We observe that the maximal minors of the following matrix vanish on $\mathcal M^{\leq 3}(G)$:
\begin{equation*} 
\begin{blockarray}{ccccccc}
 & 11 & 12 & 23 & 22 & 13 & 33 \\[2pt]
\begin{block}{c(cccccc)}
\emptyset & s_{11} & s_{12} & s_{23} & s_{22} & s_{13} & s_{33} \\
1 & t_{111} & t_{112} & t_{123} & t_{122} & t_{113} & t_{133} \\
2 & t_{112} & t_{122} & t_{223} & t_{222} & t_{123} & t_{233} \\
3 & t_{113} & t_{123} & t_{233} & t_{223} & t_{133} & t_{333} \\
\end{block}
\end{blockarray}.
\end{equation*}
In fact, 12 of the $3$-minors also vanish, although they are not enough to describe the ideal. We computed that the ideal $\mathcal{I}^{\leq 3}(G)$ is actually minimally generated by 148 polynomials: 12 cubics, 41 quartics, 67 quintics, 24 sextics and 4 septics.
It would be interesting to better understand these generators, in particular, the origin of higher degree generators.
\end{example}

The examples above motivate the further study of generators involving second and third order moments for general non-Gaussian linear structural equation models, and we expect substantial future work in this direction.

\section*{Acknowledgements}

This work is part of a project that has received funding from the European Research Council (ERC) under the European Union’s Horizon 2020 research and innovation programme (Grant agreement No.~883818).  
Elina Robeva was supported by NSERC Discovery Grant (DGECR-2020-00338).

\appendix

\section{Proof of  \Cref{th:generators}}\label{app:proof}

We provide a complete proof of \Cref{th:generators}.

\begin{proof}

\emph{Binomials in $s$.} For binomials in $\mathcal{I}^{\leq 3}(G)\cap\CC[s_{i j}]$, this follows from \cite[Theorem 5.8]{sullivant2008}.

\emph{Binomials in $t$.} Consider a binomial $f$ in $\mathcal{I}^{\leq 3}(G)\cap\CC[t_{i j k}]$. We want to prove that $f\in I\cap\CC[t_{i j k}]$.
Since $\mathcal{I}^{\leq 3}(G)$ is prime, we can assume $f$ has no common factors. A binomial of degree $r$ can be written as
$$f=t_{i j k}t_{i_2 j_2 k_2}\cdots t_{i_r j_r k_r}- t_{i' j' k'}t_{i'_2 j'_2 k'_2}\cdots t_{i'_r j'_r k'_r},$$

\noindent
where we consider the lexicographic order previously defined.
Then $i=i'$ and $\tp(i,j,k)=\tp(i,j',k')=b$, otherwise $f\notin\ker\phi_G$.

Using the standard form tableau representation this would be written as
$$f=\left[\begin{array}{l|l|l}
  \underline{b}\alpha\cdots i  &  \underline{b}\beta\cdots j &
  \underline{b}\gamma\cdots k \\
  \vdots &   \vdots &  \vdots\\
\end{array}\right]
-\left[\begin{array}{l|l|l}
  \underline{b}\alpha\cdots i   &  \underline{b}\delta\cdots j' &
  \underline{b}\varepsilon\cdots k'\\
   \vdots   &  \vdots  &  \vdots\\
\end{array}\right],$$

\noindent
where $\alpha,\beta,\gamma,\delta,\varepsilon$ denote the first nodes in each path of the 3-treks represented by variables $t_{i j k}$ and $t_{i j' k'}$ as in Figure \ref{fig:tableau}. Note that 

\begin{itemize}
    \item $\alpha\leq\beta\leq\gamma$, not all equal;
    \item $\alpha\leq\delta\leq\varepsilon$, not all equal;
    \item $\beta\leq\delta$.
\end{itemize}

\begin{figure}[h]
\centering
\begin{tikzpicture}
\node(b) at (0,1) []    {$b$};
\node(alpha) at (-2,0) []      {$\alpha$};
\node(beta) at (-1,0) []      {$\beta$};
\node(gamma) at (0,0) []      {$\gamma$};
\node(delta) at (1,0) []      {$\delta$};
\node(epsilon) at (2,0) []      {$\varepsilon$};
\node(d1) at (-2,-1) []      {$\vdots$};
\node(d2) at (-1,-1) []      {$\vdots$};
\node(d3) at (0,-1) []      {$\vdots$};
\node(d4) at (1,-1) []      {$\vdots$};
\node(d5) at (2,-1) []      {$\vdots$};
\node(i) at (-2,-2) []      {$i$};
\node(j) at (-1,-2) []      {$j$};
\node(k) at (0,-2) []      {$k$};
\node(j') at (1,-2) []      {$j'$};
\node(k') at (2,-2) []      {$k'$};
\draw (b) edge[->] (alpha) (alpha) edge[->] (d1)  (d1) edge[->] (i);
\draw[green] 
(b) edge[->] (beta) (b) edge[->] (gamma) (beta) edge[->] (d2) (gamma) edge[->] (d3) (d2) edge[->] (j) (d3) edge[->] (k);
\draw[blue] 
(b) edge[->] (delta) (b) edge[->] (epsilon) (delta) edge[->] (d4) (epsilon) edge[->] (d5) (d4) edge[->] (j') (d5) edge[->] (k');
\end{tikzpicture}
\caption{3-treks represented by $t_{i j k}$ (black and green paths) and $t_{i j' k'}$ (black and blue paths).}
\label{fig:tableau}
\end{figure}
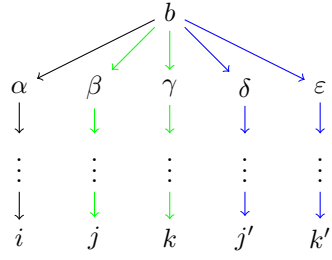

We consider two essentially different situations depending on where the first disagreement between $m_1$ and $m_2$ in $f=m_1-m_2$ takes place: either right after the top in the first row of the tableau (cases I and II) or later on (case III).

\emph{\textbf{Case I:}} Disagreement right after the top in the middle column, that is, $\beta<\delta$. Note that $\alpha\leq\beta<\delta\leq\varepsilon$. 


The string $b\beta$ must appear somewhere in $m_2$:

$$f=\left[\begin{array}{l|l|l}
  \underline{b}\alpha\cdots i  &  {\color{purple}\underline{b}\beta}\cdots j &
  \underline{b}\gamma\cdots k \\
  \vdots &   \vdots &  \vdots\\
\end{array}\right]
-\left[\begin{array}{l|l|l}
  \underline{b}\alpha\cdots i   &  \underline{b}\delta\cdots j' &
  \underline{b}\varepsilon\cdots k'\\
  \underline{b}\zeta\cdots i'' &  {\color{purple}\underline{b}\beta}\cdots j'' &
  \underline{b}\eta\cdots k''  \\
   \vdots   &  \vdots  &  \vdots\\
\end{array}\right],$$

Consider the quadratic binomial 

$$q_1-q_2=\left[\begin{array}{l|l|l}
  \underline{b}\alpha\cdots i   &  \underline{b}\delta\cdots j' &
  \underline{b}\varepsilon\cdots k'\\
  \underline{b}\zeta\cdots i'' &  {\color{purple}\underline{b}\beta}\cdots j'' &
  \underline{b}\eta\cdots k''  
\end{array}\right]
-\left[\begin{array}{l|l|l}
  \underline{b}\alpha\cdots i   &  {\color{purple}\underline{b}\beta}\cdots j'' &
  \underline{b}\varepsilon\cdots k'\\
  \underline{b}\zeta\cdots i'' &  \underline{b}\delta\cdots j' &
  \underline{b}\eta\cdots k''  
\end{array}\right],$$

\noindent
with $m_2=q_1m$.

Under certain conditions (namely $\tp(i,j'',k')=\tp(i'',j',k'')=b$), this tableau represents the polynomial $q=t_{i j' k'}t_{i'' j'' k''}-t_{i j'' k'}t_{i'' j' k''}\in\mathcal{I}^{\leq 3}(G)$. 
These conditions are necessary because otherwise the rows in the tableau do not correspond to variables in $\CC[t_{i j k}]$.

If such conditions are satisfied, then the polynomial
$$f':=m_1-q_2m=m_1-m_2+(q_1-q_2)m=f+qm\in\mathcal{I}^{\leq 3}(G)$$
has fewer disagreements in the first row of its corresponding tableau.
The quadatic binomial $q$ arises as a 2-minor of $A_{j',j''}$, hence we can think of this binomial move as swaping the path from the top to $j'$ in the 3-trek represented by $t_{i j' k'}$ and the path from the same top to $j''$ in $t_{i'' j'' k''}$:

\begin{figure}[h]
\centering
\begin{tikzpicture}
\node(b) at (0,1) []    {$b$};
\node(alpha) at (-2.5,0) []      {$\alpha$};
\node(delta) at (-1.5,0) []      {$\delta$};
\node(epsilon) at (-0.5,0) []      {$\varepsilon$};
\node(zeta) at (0.5,0) []      {$\zeta$};
\node(beta) at (1.5,0) []      {$\beta$};
\node(eta) at (2.5,0) []      {$\eta$};
\node(d1) at (-2.5,-1) []      {$\vdots$};
\node(d2) at (-1.5,-1) []      {$\vdots$};
\node(d3) at (-0.5,-1) []      {$\vdots$};
\node(d4) at (0.5,-1) []      {$\vdots$};
\node(d5) at (1.5,-1) []      {$\vdots$};
\node(d6) at (2.5,-1) []      {$\vdots$};
\node(i) at (-2.5,-2) []      {$i$};
\node(j') at (-1.5,-2) []      {$j'$};
\node(k') at (-0.5,-2) []      {$k'$};
\node(i'') at (0.5,-2) []      {$i''$};
\node(j'') at (1.5,-2) []      {$j''$};
\node(k'') at (2.5,-2) []      {$k''$};
\draw[green] 
(b) edge[->] (alpha)
(b) edge[->] (delta) 
(b) edge[->] (epsilon) 
(alpha) edge[->] (d1) 
(delta) edge[->] (d2) 
(epsilon) edge[->] (d3) 
(d1) edge[->] (i) 
(d2) edge[->] (j') 
(d3) edge[->] (k');
\draw[blue] 
(b) edge[->] (zeta)
(b) edge[->] (beta) 
(b) edge[->] (eta) 
(zeta) edge[->] (d4) 
(beta) edge[->] (d5) 
(eta) edge[->] (d6) 
(d4) edge[->] (i'') 
(d5) edge[->] (j'') 
(d6) edge[->] (k'');
\end{tikzpicture}\quad\quad\begin{tikzpicture}
\node(b) at (0,1) []    {$b$};
\node(alpha) at (-2.5,0) []      {$\alpha$};
\node(delta) at (-1.5,0) []      {$\delta$};
\node(epsilon) at (-0.5,0) []      {$\varepsilon$};
\node(zeta) at (0.5,0) []      {$\zeta$};
\node(beta) at (1.5,0) []      {$\beta$};
\node(eta) at (2.5,0) []      {$\eta$};
\node(d1) at (-2.5,-1) []      {$\vdots$};
\node(d2) at (-1.5,-1) []      {$\vdots$};
\node(d3) at (-0.5,-1) []      {$\vdots$};
\node(d4) at (0.5,-1) []      {$\vdots$};
\node(d5) at (1.5,-1) []      {$\vdots$};
\node(d6) at (2.5,-1) []      {$\vdots$};
\node(i) at (-2.5,-2) []      {$i$};
\node(j') at (-1.5,-2) []      {$j'$};
\node(k') at (-0.5,-2) []      {$k'$};
\node(i'') at (0.5,-2) []      {$i''$};
\node(j'') at (1.5,-2) []      {$j''$};
\node(k'') at (2.5,-2) []      {$k''$};
\draw[green] 
(b) edge[->] (alpha)
(b) edge[->] (epsilon) 
(b) edge[->] (beta)
(alpha) edge[->] (d1) 
(beta) edge[->] (d5) 
(epsilon) edge[->] (d3) 
(d1) edge[->] (i) 
(d3) edge[->] (k')
(d5) edge[->] (j'');
\draw[blue] 
(b) edge[->] (delta) 
(b) edge[->] (zeta) 
(b) edge[->] (eta) 
(delta) edge[->] (d2) 
(zeta) edge[->] (d4) 
(eta) edge[->] (d6) 
(d2) edge[->] (j') 
(d4) edge[->] (i'') 
(d6) edge[->] (k'');
\end{tikzpicture}
\caption{Swap of paths between two 3-treks encoded by a 2-minor of $A_{j',j''}$.}
\label{fig:swap}
\end{figure}
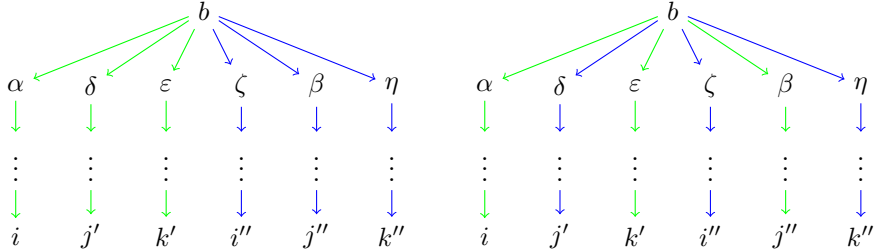

If $\tp(i,j'',k')=\tp(i'',j',k'')=b$ does not hold, then this move can not be applied.
This situation only occurs if either $\alpha=\beta=\varepsilon$ or $\delta=\zeta=\eta$. 

Since $\alpha<\varepsilon$ by assumption, we only need to study the case $\delta=\zeta=\eta$. Then 

$$f=\left[\begin{array}{l|l|l}
  \underline{b}\alpha\cdots i  &  {\color{purple}\underline{b}\beta}\cdots j &
  \underline{b}\gamma\cdots k \\
  \vdots &   \vdots &  \vdots\\
\end{array}\right]
-\left[\begin{array}{l|l|l}
  \underline{b}\alpha\cdots i   &  \underline{b}\delta\cdots j' &
  \underline{b}\varepsilon\cdots k'\\
  \underline{b}\delta\cdots i'' &  {\color{purple}\underline{b}\beta}\cdots j'' &
  \underline{b}\delta\cdots k''  \\
   \vdots   &  \vdots  &  \vdots\\
\end{array}\right].$$

Under suitable conditions that will be made precise as we move on, we can perform two quadratic moves that will allow us to reduce the disagreements in the first row.

If $\delta\neq\gamma$, there exists a row in $m_2$ with top $b$ with at most one occurrence of the pattern $b\delta$. Indeed, assume that $f$ has $r+1$ rows with top $b$. By construction, the first row of $m_1$ has no occurrences of $b\delta$ whereas the first of row of $m_2$ has at least one occurrence of $b\delta$. Note that $m_1$ can have at most $2r$ instances of $b\delta$ overall. If all the remaining $r$ rows of $m_2$ have 2 occurrences of $b\delta$, there would be at least $2r+1$ instances of $b\delta$, which poses a contradiction.
Therefore, we can write our binomial as

\begin{equation}\label{eq:specialBinomial}
f=\left[\begin{array}{l|l|l}
  \underline{b}\alpha\cdots i  &  {\color{purple}\underline{b}\beta}\cdots j &
  \underline{b}\gamma\cdots k \\
  \vdots &   \vdots &  \vdots\\
\end{array}\right]
-\left[\begin{array}{l|l|l}
  \underline{b}\alpha\cdots i   &  \underline{b}\delta\cdots j' &
  \underline{b}\varepsilon\cdots k'\\
  \underline{b}\delta\cdots i'' &  {\color{purple}\underline{b}\beta}\cdots j'' &
  \underline{b}\delta\cdots k''  \\
  \underline{b}\alpha'\cdots i''' &  
  \underline{b}\beta'\cdots j''' &
  \underline{b}\gamma'\cdots k''' \\
   \vdots   &  \vdots  &  \vdots\\
\end{array}\right],
\end{equation}

\noindent
where $\alpha',\beta',\gamma'$ are not all equal and at most one of them is equal to $\beta$. 

What are the quadratic moves that reduce disagreements in the first row of $m_2$ and which are the conditions on $\alpha',\beta',\gamma'$ that make these moves possible?

$$q_{j'',j'''}=\left[\begin{array}{l|l|l}
 \underline{b}\delta\cdots i'' &  {\color{purple}\underline{b}\beta}\cdots j'' &
  \underline{b}\delta\cdots k''  \\
  \underline{b}\alpha'\cdots i''' &  
  \underline{b}\beta'\cdots j''' &
  \underline{b}\gamma'\cdots k''' \\
\end{array}\right]
-\left[\begin{array}{l|l|l}
 \underline{b}\delta\cdots i'' &  \underline{b}\beta'\cdots j''' &
  \underline{b}\delta\cdots k''  \\
  \underline{b}\alpha'\cdots i''' &  
  {\color{purple}\underline{b}\beta}\cdots j'' &
  \underline{b}\gamma'\cdots k''' \\
\end{array}\right]$$

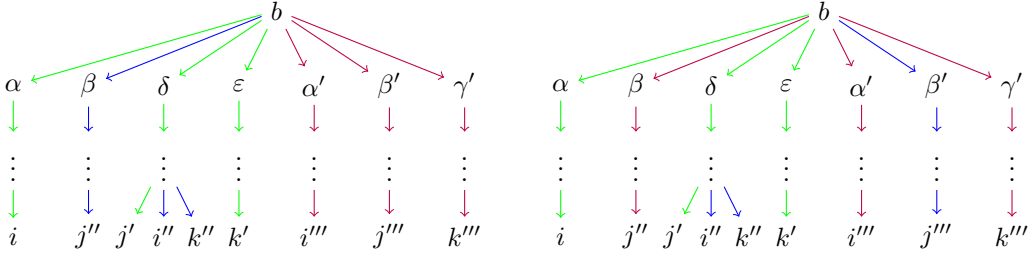
\begin{figure}[h]
\centering
\begin{tikzpicture}
\node(b) at (0,1) []    {$b$};
\node(alpha) at (-3.5,0) []      {$\alpha$};
\node(beta) at (-2.5,0) []      {$\beta$};
\node(delta) at (-1.5,0) []      {$\delta$};
\node(epsilon) at (-0.5,0) []      {$\varepsilon$};
\node(zeta) at (0.5,0) []      {$\alpha'$};
\node(betap) at (1.5,0) []      {$\beta'$};
\node(eta) at (2.5,0) []      {$\gamma'$};
\node(d0) at (-3.5,-1) []      {$\vdots$};
\node(d1) at (-2.5,-1) []      {$\vdots$};
\node(d2) at (-1.5,-1) []      {$\vdots$};
\node(d3) at (-0.5,-1) []      {$\vdots$};
\node(d4) at (0.5,-1) []      {$\vdots$};
\node(d5) at (1.5,-1) []      {$\vdots$};
\node(d6) at (2.5,-1) []      {$\vdots$};
\node(i) at (-3.5,-2) []      {$i$};
\node(j'') at (-2.5,-2) []      {$j''$};
\node(j') at (-2,-2) []      {$j'$};
\node(i'') at (-1.5,-2) []      {$i''$};
\node(k'') at (-1,-2) []      {$k''$};
\node(k') at (-0.5,-2) []      {$k'$};
\node(i''') at (0.5,-2) []      {$i'''$};
\node(j''') at (1.5,-2) []      {$j'''$};
\node(k''') at (2.5,-2) []      {$k'''$};
\draw[green] 
(b) edge[->] (alpha)
(b) edge[->] (delta) 
(b) edge[->] (epsilon) 
(alpha) edge[->] (d0) 
(delta) edge[->] (d2) 
(epsilon) edge[->] (d3) 
(d0) edge[->] (i) 
(d2) edge[->] (j') 
(d3) edge[->] (k');
\draw[blue]
(b) edge[->] (beta)
(beta) edge[->] (d1)
(d1) edge[->] (j'')
(d2) edge[->] (i'')
(d2) edge[->] (k'');
\draw[purple] 
(b) edge[->] (zeta)
(b) edge[->] (betap) 
(b) edge[->] (eta) 
(zeta) edge[->] (d4) 
(betap) edge[->] (d5) 
(eta) edge[->] (d6) 
(d4) edge[->] (i''') 
(d5) edge[->] (j''') 
(d6) edge[->] (k''');
\end{tikzpicture}\quad\quad\begin{tikzpicture}
\node(b) at (0,1) []    {$b$};
\node(alpha) at (-3.5,0) []      {$\alpha$};
\node(beta) at (-2.5,0) []      {$\beta$};
\node(delta) at (-1.5,0) []      {$\delta$};
\node(epsilon) at (-0.5,0) []      {$\varepsilon$};
\node(zeta) at (0.5,0) []      {$\alpha'$};
\node(betap) at (1.5,0) []      {$\beta'$};
\node(eta) at (2.5,0) []      {$\gamma'$};
\node(d0) at (-3.5,-1) []      {$\vdots$};
\node(d1) at (-2.5,-1) []      {$\vdots$};
\node(d2) at (-1.5,-1) []      {$\vdots$};
\node(d3) at (-0.5,-1) []      {$\vdots$};
\node(d4) at (0.5,-1) []      {$\vdots$};
\node(d5) at (1.5,-1) []      {$\vdots$};
\node(d6) at (2.5,-1) []      {$\vdots$};
\node(i) at (-3.5,-2) []      {$i$};
\node(j'') at (-2.5,-2) []      {$j''$};
\node(j') at (-2,-2) []      {$j'$};
\node(i'') at (-1.5,-2) []      {$i''$};
\node(k'') at (-1,-2) []      {$k''$};
\node(k') at (-0.5,-2) []      {$k'$};
\node(i''') at (0.5,-2) []      {$i'''$};
\node(j''') at (1.5,-2) []      {$j'''$};
\node(k''') at (2.5,-2) []      {$k'''$};
\draw[green] 
(b) edge[->] (alpha)
(b) edge[->] (delta) 
(b) edge[->] (epsilon) 
(alpha) edge[->] (d0) 
(delta) edge[->] (d2) 
(epsilon) edge[->] (d3) 
(d0) edge[->] (i) 
(d2) edge[->] (j') 
(d3) edge[->] (k');
\draw[blue]
(b) edge[->] (betap)
(betap) edge[->] (d5) 
(d5) edge[->] (j''') 
(d2) edge[->] (i'')
(d2) edge[->] (k'');
\draw[purple] 
(b) edge[->] (beta)
(beta) edge[->] (d1)
(d1) edge[->] (j'')
(b) edge[->] (zeta)
(b) edge[->] (eta) 
(zeta) edge[->] (d4) 
(eta) edge[->] (d6) 
(d4) edge[->] (i''') 
(d6) edge[->] (k''');
\end{tikzpicture}
\caption{The quadratic move is only applicable iff $\delta\neq\beta'$ and $\beta,\alpha',\delta'$ are not all equal.}
\label{fig:TwoMoves1}
\end{figure}

$$q_{j',j'''}=\left[\begin{array}{l|l|l}
 \underline{b}\alpha\cdots i &  \underline{b}\delta\cdots j' &
  \underline{b}\varepsilon\cdots k'  \\
 \underline{b}\alpha'\cdots i''' &  
  {\color{purple}\underline{b}\beta}\cdots j'' &
  \underline{b}\gamma'\cdots k''' \\
\end{array}\right]
-\left[\begin{array}{l|l|l}
 \ \underline{b}\alpha\cdots i &   {\color{purple}\underline{b}\beta}\cdots j'' &
  \underline{b}\varepsilon\cdots k'  \\
 \underline{b}\alpha'\cdots i''' &  
 \underline{b}\delta\cdots j'
  &
  \underline{b}\gamma'\cdots k''' \\
\end{array}\right]$$

\begin{figure}[h]
\centering
\begin{tikzpicture}
\node(b) at (0,1) []    {$b$};
\node(alpha) at (-3.5,0) []      {$\alpha$};
\node(beta) at (-2.5,0) []      {$\beta$};
\node(delta) at (-1.5,0) []      {$\delta$};
\node(epsilon) at (-0.5,0) []      {$\varepsilon$};
\node(zeta) at (0.5,0) []      {$\alpha'$};
\node(betap) at (1.5,0) []      {$\beta'$};
\node(eta) at (2.5,0) []      {$\gamma'$};
\node(d0) at (-3.5,-1) []      {$\vdots$};
\node(d1) at (-2.5,-1) []      {$\vdots$};
\node(d2) at (-1.5,-1) []      {$\vdots$};
\node(d3) at (-0.5,-1) []      {$\vdots$};
\node(d4) at (0.5,-1) []      {$\vdots$};
\node(d5) at (1.5,-1) []      {$\vdots$};
\node(d6) at (2.5,-1) []      {$\vdots$};
\node(i) at (-3.5,-2) []      {$i$};
\node(j'') at (-2.5,-2) []      {$j''$};
\node(j') at (-2,-2) []      {$j'$};
\node(i'') at (-1.5,-2) []      {$i''$};
\node(k'') at (-1,-2) []      {$k''$};
\node(k') at (-0.5,-2) []      {$k'$};
\node(i''') at (0.5,-2) []      {$i'''$};
\node(j''') at (1.5,-2) []      {$j'''$};
\node(k''') at (2.5,-2) []      {$k'''$};
\draw[green] 
(b) edge[->] (alpha)
(b) edge[->] (delta) 
(b) edge[->] (epsilon) 
(alpha) edge[->] (d0) 
(delta) edge[->] (d2) 
(epsilon) edge[->] (d3) 
(d0) edge[->] (i) 
(d2) edge[->] (j') 
(d3) edge[->] (k');
\draw[blue]
(b) edge[->] (betap)
(betap) edge[->] (d5) 
(d5) edge[->] (j''') 
(d2) edge[->] (i'')
(d2) edge[->] (k'');
\draw[purple] 
(b) edge[->] (beta)
(beta) edge[->] (d1)
(d1) edge[->] (j'')
(b) edge[->] (zeta)
(b) edge[->] (eta) 
(zeta) edge[->] (d4) 
(eta) edge[->] (d6) 
(d4) edge[->] (i''') 
(d6) edge[->] (k''');
\end{tikzpicture}\quad\quad
\begin{tikzpicture}
\node(b) at (0,1) []    {$b$};
\node(alpha) at (-3.5,0) []      {$\alpha$};
\node(beta) at (-2.5,0) []      {$\beta$};
\node(delta) at (-1.5,0) []      {$\delta$};
\node(epsilon) at (-0.5,0) []      {$\varepsilon$};
\node(zeta) at (0.5,0) []      {$\alpha'$};
\node(betap) at (1.5,0) []      {$\beta'$};
\node(eta) at (2.5,0) []      {$\gamma'$};
\node(d0) at (-3.5,-1) []      {$\vdots$};
\node(d1) at (-2.5,-1) []      {$\vdots$};
\node(d2) at (-1.5,-1) []      {$\vdots$};
\node(d3) at (-0.5,-1) []      {$\vdots$};
\node(d4) at (0.5,-1) []      {$\vdots$};
\node(d5) at (1.5,-1) []      {$\vdots$};
\node(d6) at (2.5,-1) []      {$\vdots$};
\node(i) at (-3.5,-2) []      {$i$};
\node(j'') at (-2.5,-2) []      {$j''$};
\node(j') at (-2,-2) []      {$j'$};
\node(i'') at (-1.5,-2) []      {$i''$};
\node(k'') at (-1,-2) []      {$k''$};
\node(k') at (-0.5,-2) []      {$k'$};
\node(i''') at (0.5,-2) []      {$i'''$};
\node(j''') at (1.5,-2) []      {$j'''$};
\node(k''') at (2.5,-2) []      {$k'''$};
\draw[green] 
(b) edge[->] (alpha)
(b) edge[->] (beta)
(beta) edge[->] (d1)
(d1) edge[->] (j'')
(b) edge[->] (epsilon) 
(alpha) edge[->] (d0) 
(epsilon) edge[->] (d3) 
(d0) edge[->] (i) 
(d3) edge[->] (k');
\draw[blue]
(b) edge[->] (betap)
(betap) edge[->] (d5) 
(d5) edge[->] (j''') 
(d2) edge[->] (i'')
(d2) edge[->] (k'');
\draw[purple] 
(b) edge[->] (delta) 
(b) edge[->] (zeta)
(b) edge[->] (eta) 
(delta) edge[->] (d2) 
(zeta) edge[->] (d4) 
(eta) edge[->] (d6) 
(d2) edge[->] (j') 
(d4) edge[->] (i''') 
(d6) edge[->] (k''');
\end{tikzpicture}
\caption{The move is only applicable iff $\delta,\alpha',\delta'$ are not all equal, since $\alpha=\beta=\varepsilon$ is not possible by assumption.}
\label{fig:TwoMoves2}
\end{figure}
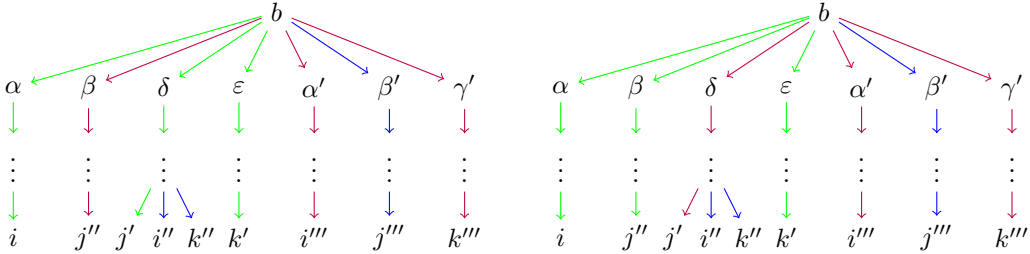

The conditions $\delta\neq\beta'$ (from Figure \ref{fig:TwoMoves1}) and $\delta,\alpha',\gamma'$ not all equal (from Figure \ref{fig:TwoMoves2}) needed in order to apply the quadratic moves can be translated into requiring that the pattern $b\delta$ appears at most once in the third row of $m_2$. Hence there is a single situation where we cannot apply the quadratic moves: $\beta=\alpha'=\gamma'$ (from Figure \ref{fig:TwoMoves1}).

Assume that $\beta=\alpha'=\gamma'$. In the second term of $f$ in (\ref{eq:specialBinomial}), we can now use the third row to reduce the disagreements in the first one provided $\beta=\delta=\beta'$ does not hold, see Figure \ref{fig:UsingRow3}.

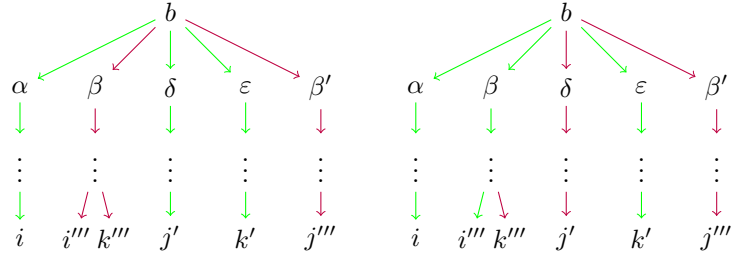
\begin{figure}[h]
\centering
\begin{tikzpicture}
\node(b) at (0,1) []    {$b$};
\node(alpha) at (-2,0) []      {$\alpha$};
\node(beta) at (-1,0) []      {$\beta$};
\node(delta) at (0,0) []      {$\delta$};
\node(epsilon) at (1,0) []      {$\varepsilon$};
\node(betap) at (2,0) []      {$\beta'$};
\node(d0) at (-2,-1) []      {$\vdots$};
\node(d1) at (-1,-1) []      {$\vdots$};
\node(d2) at (0,-1) []      {$\vdots$};
\node(d3) at (1,-1) []      {$\vdots$};
\node(d4) at (2,-1) []      {$\vdots$};
\node(i) at (-2,-2) []      {$i$};
\node(i''') at (-1.25,-2) []      {$i'''$};
\node(k''') at (-0.75,-2) []      {$k'''$};
\node(j') at (0,-2) []      {$j'$};
\node(k') at (1,-2) []      {$k'$};
\node(j''') at (2,-2) []      {$j'''$};
\draw[green] 
(b) edge[->] (alpha)
(b) edge[->] (delta) 
(b) edge[->] (epsilon) 
(alpha) edge[->] (d0) 
(delta) edge[->] (d2) 
(epsilon) edge[->] (d3) 
(d0) edge[->] (i) 
(d2) edge[->] (j') 
(d3) edge[->] (k');
\draw[purple] 
(b) edge[->] (beta)
(beta) edge[->] (d1)
(d1) edge[->] (i''')
(d1) edge[->] (k''')
(b) edge[->] (betap) 
(betap) edge[->] (d4) 
(d4) edge[->] (j''');
\end{tikzpicture}\quad\quad\begin{tikzpicture}
\node(b) at (0,1) []    {$b$};
\node(alpha) at (-2,0) []      {$\alpha$};
\node(beta) at (-1,0) []      {$\beta$};
\node(delta) at (0,0) []      {$\delta$};
\node(epsilon) at (1,0) []      {$\varepsilon$};
\node(betap) at (2,0) []      {$\beta'$};
\node(d0) at (-2,-1) []      {$\vdots$};
\node(d1) at (-1,-1) []      {$\vdots$};
\node(d2) at (0,-1) []      {$\vdots$};
\node(d3) at (1,-1) []      {$\vdots$};
\node(d4) at (2,-1) []      {$\vdots$};
\node(i) at (-2,-2) []      {$i$};
\node(i''') at (-1.25,-2) []      {$i'''$};
\node(k''') at (-0.75,-2) []      {$k'''$};
\node(j') at (0,-2) []      {$j'$};
\node(k') at (1,-2) []      {$k'$};
\node(j''') at (2,-2) []      {$j'''$};
\draw[green] 
(b) edge[->] (alpha)
(b) edge[->] (epsilon) 
(b) edge[->] (beta)
(beta) edge[->] (d1)
(d1) edge[->] (i''')
(alpha) edge[->] (d0) 
(epsilon) edge[->] (d3) 
(d0) edge[->] (i) 
(d3) edge[->] (k');
\draw[purple] 
(b) edge[->] (delta) 
(delta) edge[->] (d2) 
(d2) edge[->] (j')
(d1) edge[->] (k''')
(b) edge[->] (betap) 
(betap) edge[->] (d4) 
(d4) edge[->] (j''');
\end{tikzpicture}
\caption{This quadratic move can be applied unless $\beta=\delta=\beta'$.}
\label{fig:UsingRow3}
\end{figure}

But this "bad" scenario yields $\beta=\delta=\beta'=\alpha'=\gamma'$, which can never occur, otherwise the third row of $m_2$ would not represent a variable in $\CC[t_{ijk}]$.

On the other hand, if $\gamma=\delta$, the pattern $b\beta$ must appear somewhere in the second term of $f$
$$f=\left[\begin{array}{l|l|l}
  \underline{b}\alpha\cdots i  &  {\color{purple}\underline{b}\beta}\cdots j &
  \underline{b}\gamma\cdots k \\
  \vdots &   \vdots &  \vdots\\
\end{array}\right]
-\left[\begin{array}{l|l|l}
  \underline{b}\alpha\cdots i   &  \underline{b}\gamma\cdots j' &
  \underline{b}\varepsilon\cdots k'\\
  \underline{b}\zeta\cdots i'' &  \underline{b}\eta\cdots j' &
  {\color{purple}\underline{b}\beta}\cdots k''\\
   \vdots   &  \vdots  &  \vdots\\
\end{array}\right]$$
and can be used to reduce the disagreements between $b\beta$ and $b\varepsilon$ as long as $\zeta=\eta=\varepsilon$ does not hold. Otherwise
$$f=\left[\begin{array}{l|l|l}
  \underline{b}\alpha\cdots i  &  {\color{purple}\underline{b}\beta}\cdots j &
  \underline{b}\gamma\cdots k \\
  \vdots &   \vdots &  \vdots\\
\end{array}\right]
-\left[\begin{array}{l|l|l}
  \underline{b}\alpha\cdots i   &  \underline{b}\gamma\cdots j' &
  \underline{b}\varepsilon\cdots k'\\
  \underline{b}\varepsilon\cdots i'' &  \underline{b}\varepsilon\cdots j'' &
  {\color{purple}\underline{b}\beta}\cdots k''\\ \underline{b}\alpha'\cdots i''' &  
 \underline{b}\beta'\cdots j'''
  &
  \underline{b}\gamma'\cdots k''' \\
   \vdots   &  \vdots  &  \vdots\\
\end{array}\right]$$
and there is a third row in $m_2$ with at most one occurrence of $b\varepsilon$. Indeed, note that the first row of $m_1$ always has one less occurrence of $b\varepsilon$ than the first row of $m_2$, regardless of whether $\gamma=\varepsilon$ or not. As long as $\alpha'=\beta'=\varepsilon$ does not hold, two quadratic moves as in Figures \ref{fig:TwoMoves1} and \ref{fig:TwoMoves2} reduce the disagreements. 

The remaining case $\alpha'=\beta'=\varepsilon$ corresponds to 
$$f=\left[\begin{array}{l|l|l}
  \underline{b}\alpha\cdots i  &  {\color{purple}\underline{b}\beta}\cdots j &
  \underline{b}\gamma\cdots k \\
  \vdots &   \vdots &  \vdots\\
\end{array}\right]
-\left[\begin{array}{l|l|l}
  \underline{b}\alpha\cdots i   &  \underline{b}\gamma\cdots j' &
  \underline{b}\varepsilon\cdots k'\\
  \underline{b}\varepsilon\cdots i'' &  \underline{b}\varepsilon\cdots j'' &
  {\color{purple}\underline{b}\beta}\cdots k''\\ {\color{purple}\underline{b}\beta}\cdots i''' &  
 {\color{purple}\underline{b}\beta}\cdots j'''
  &
  \underline{b}\gamma'\cdots k''' \\
   \vdots   &  \vdots  &  \vdots\\
\end{array}\right].$$

Note that once the existence of the third row is ensured, we no longer have to worry about the lexicographic ordering, hence we omit the second row of $m_2$ and rewrite the third one as

$$f=\left[\begin{array}{l|l|l}
  \underline{b}\alpha\cdots i  &  {\color{purple}\underline{b}\beta}\cdots j &
  \underline{b}\gamma\cdots k \\
  \vdots &   \vdots &  \vdots\\
\end{array}\right]
-\left[\begin{array}{l|l|l}
  \underline{b}\alpha\cdots i   &  \underline{b}\gamma\cdots j' &
  \underline{b}\varepsilon\cdots k'\\
  \underline{b}\gamma'\cdots k''' &
 {\color{purple}\underline{b}\beta}\cdots j'''
  &
   {\color{purple}\underline{b}\beta}\cdots i''' \\
   \vdots   &  \vdots  &  \vdots\\
\end{array}\right].$$

A single quadratic move reduces disagreements provided $\gamma'=\beta=\varepsilon$ does not hold. But note that this last equality would yield $\alpha'=\beta'=\gamma'$, which is a contradiction and hence the move is always allowed.

\emph{\textbf{Case II:}} Disagreement right after the top only in the third column, that is $\beta=\delta$ and $\gamma<\varepsilon$.

The string $b\gamma$ must appear somewhere in the second term of $f$:

$$f=\left[\begin{array}{l|l|l}
  \underline{b}\alpha\cdots i  &  \
  \underline{b}\beta\cdots j &
  {\color{purple}\underline{b}\gamma}\cdots k \\
  \vdots &   \vdots &  \vdots\\
\end{array}\right]
-\left[\begin{array}{l|l|l}
  \underline{b}\alpha\cdots i   &  \underline{b}\beta\cdots j' &
  \underline{b}\varepsilon\cdots k'\\
  \underline{b}\zeta\cdots i'' &  \underline{b}\eta\cdots j'' &
  {\color{purple}\underline{b}\gamma}\cdots k''\\
   \vdots   &  \vdots  &  \vdots\\
\end{array}\right],$$

If $\tp(i,j',k'')=\tp(i'',j'',k')$, we can apply the quadratic binomial move 

$$\left[\begin{array}{l|l|l}
  \underline{b}\alpha\cdots i   &  
  \underline{b}\beta\cdots j' &
  \underline{b}\varepsilon\cdots k'\\
  \underline{b}\zeta\cdots i'' &  
  \underline{b}\eta\cdots j'' &
  {\color{purple}\underline{b}\gamma}\cdots k''  
\end{array}\right]
-\left[\begin{array}{l|l|l}
  \underline{b}\alpha\cdots i   &  
  \underline{b}\beta\cdots j' &
  {\color{purple}\underline{b}\gamma}\cdots k''\\
  \underline{b}\zeta\cdots i'' & 
  \underline{b}\eta\cdots j'' &
  \underline{b}\varepsilon\cdots k'  
\end{array}\right]$$

\noindent
in order to reduce the number of disagreements. 

If there is no equality between the tops, then either $\alpha=\beta=\gamma$ or $\zeta=\eta=\varepsilon$. The first option is actually not possible because otherwise the first row of $m_1$ would not be a proper variable in $\CC[t_{i j k}]$, hence

$$f=\left[\begin{array}{l|l|l}
  \underline{b}\alpha\cdots i  &  \
  \underline{b}\beta\cdots j &
  {\color{purple}\underline{b}\gamma}\cdots k \\
  \vdots &   \vdots &  \vdots\\
\end{array}\right]
-\left[\begin{array}{l|l|l}
  \underline{b}\alpha\cdots i   &  \underline{b}\beta\cdots j' &
  \underline{b}\varepsilon\cdots k'\\
  \underline{b}\varepsilon\cdots i'' & 
  \underline{b}\varepsilon\cdots j'' &
  {\color{purple}\underline{b}\gamma}\cdots k''\\
   \vdots   &  \vdots  &  \vdots\\
\end{array}\right].$$

Since $\varepsilon\neq\gamma$ by assumption, there must be a third row with top $b$ with at most one occurrence of the pattern $b\varepsilon$ (see Case I):

$$f=\left[\begin{array}{l|l|l}
  \underline{b}\alpha\cdots i  &  \
  \underline{b}\beta\cdots j &
  {\color{purple}\underline{b}\gamma}\cdots k \\
  \vdots &   \vdots &  \vdots\\
\end{array}\right]
-\left[\begin{array}{l|l|l}
  \underline{b}\alpha\cdots i   &  \underline{b}\beta\cdots j' &
  \underline{b}\varepsilon\cdots k'\\
  \underline{b}\varepsilon\cdots i'' & 
  \underline{b}\varepsilon\cdots j'' &
  {\color{purple}\underline{b}\gamma}\cdots k''\\
  \underline{b}\alpha'\cdots i''' &  
  \underline{b}\beta'\cdots j''' &
  \underline{b}\gamma'\cdots k''' \\
   \vdots   &  \vdots  &  \vdots\\
\end{array}\right].$$

We can perform two quadratic moves as in Figure \ref{fig:TwoMoves1} and Figure \ref{fig:TwoMoves2} to reduce disagreements iff the following does not occur:
\begin{itemize}
    \item $\gamma'=\varepsilon$ or $\alpha'=\beta'=\gamma$;
    \item $\alpha'=\beta'=\varepsilon$.
\end{itemize}

As in Case I, not having $\gamma'=\varepsilon$ nor $\alpha'=\beta'=\varepsilon$ can be translated into having at most one occurrence of $b\varepsilon$ in the third row of $m_2$. An analogous argument to Case I shows that when $\alpha'=\beta'=\gamma$ we can always reduce disagreements in the first row using only the third row.

This ends the discussion on how to reduce disagreements whenever they come right after the top.

\emph{Case III:} Let us now assume that the first disagreement comes later on. Let $A$ be a string of characters such that

$$f=\left[\begin{array}{l|l|l}
  \underline{b}\alpha\cdots   &  \underline{b}B{\color{purple}x\beta}\cdots&
  \underline{b}\gamma\cdots \\
  \vdots &   \vdots &  \vdots\\
\end{array}\right]
-\left[\begin{array}{l|l|l}
  \underline{b}\alpha\cdots    &  \underline{b}Bx\beta'\cdots &
  \underline{b}\delta\cdots \\
   \vdots   &  \vdots  &  \vdots\\
\end{array}\right]$$

\noindent 
is in lexicographic ordering, namely $\beta <\beta'$ and the pattern $x\beta$ must appear in another row of $m_2$.

If $x\beta$ appears in a row where its top is $b'\neq x$, then there always exists a quadratic move

$$q=\left[\begin{array}{l|l|l}
  \underline{b}\alpha\cdots    &  \underline{b}Bx\beta'\cdots &
  \underline{b}\gamma\cdots \\
  \underline{b'}\varepsilon\cdots  &  \underline{b'}B'{\color{purple}x\beta}\cdots  &  \underline{b'}\eta\cdots\\
\end{array}\right]-
\left[\begin{array}{l|l|l}
  \underline{b}\alpha\cdots    &  \underline{b}B{\color{purple}x\beta}\cdots &
  \underline{b}\delta\cdots \\
  \underline{b'}\varepsilon\cdots     &  \underline{b'}B'x\beta'\cdots  &  \underline{b'}\eta\cdots\\
\end{array}\right]$$

that will reduce disagreements in the first row.

\begin{figure}[h]
\centering
\begin{tikzpicture}
\node(b) at (1,2) []    {$b$};
\node(A) at (-1,1.5) []      {$\vdots$};
\node(b') at (-2,1) []      {$b'$};
\node(D) at (1,1) []      {$\alpha$};
\node(E) at (2,1) []      {$\gamma$};
\node(A') at (-2,0) []      {$\vdots$};
\node(x) at (-2,-1) []      {$x$};
\node(alpha') at (-3,-2) []      {$\beta'$};
\node(alpha) at (-2,-2) []      {$\beta$};
\node(eta) at (-1,-2) []      {$\varepsilon$};
\node(zeta) at (0,-2) []      {$\eta$};
\draw[green] 
(b) edge[->] (A)
(A) edge[->] (b')
(b) edge[->] (D) 
(b) edge[->] (E)
(b') edge[->] (A')
(A') edge[->] (x)
(x) edge[->] (alpha');
\draw[blue]
(b') edge[->] (eta)
(b') edge[->] (zeta) 
(x) edge[->] (alpha); 
\end{tikzpicture}\quad\quad\begin{tikzpicture}
\node(b) at (1,2) []    {$b$};
\node(A) at (-1,1.5) []      {$\vdots$};
\node(b') at (-2,1) []      {$b'$};
\node(D) at (1,1) []      {$\alpha$};
\node(E) at (2,1) []      {$\delta$};
\node(A') at (-2,0) []      {$\vdots$};
\node(x) at (-2,-1) []      {$x$};
\node(alpha') at (-3,-2) []      {$\beta'$};
\node(alpha) at (-2,-2) []      {$\beta$};
\node(eta) at (-1,-2) []      {$\varepsilon$};
\node(zeta) at (0,-2) []      {$\eta$};
\draw[green] 
(b) edge[->] (A)
(A) edge[->] (b')
(b) edge[->] (D) 
(b) edge[->] (E)
(b') edge[->] (A')
(A') edge[->] (x)
(x) edge[->] (alpha);
\draw[blue]
(b') edge[->] (eta)
(b') edge[->] (zeta) 
(x) edge[->] (alpha'); 
\end{tikzpicture}
\caption{This quadratic move can always be applied because $\tp(\beta',\varepsilon,\eta)=\tp(\beta,\varepsilon,\eta)=b'$ and $\tp(\alpha,\beta',\delta)=\tp(\alpha,\beta,\delta)=b$ by construction. Note that these are partial moves, that is, they do not correspond directly to 2-minors of matrices $A_{i,j}$.
}
\label{fig:DisagreementsLaterOn}
\end{figure}
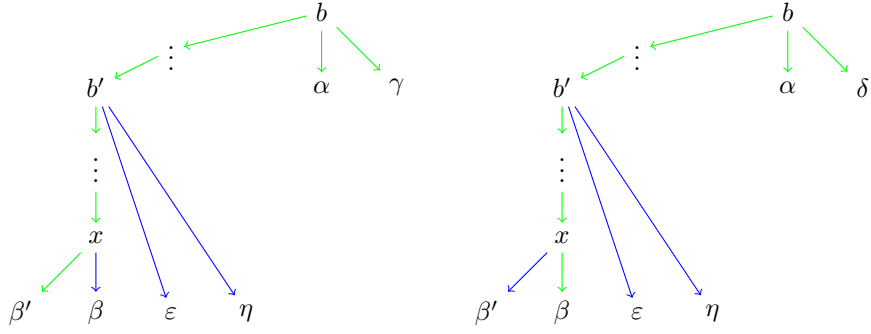

If $b'=x$, then a similar quadratic move can be applied provided we are not in the situation $\beta'=\varepsilon=\eta$. In such case, we would have

$$f=\left[\begin{array}{l|l|l}
  \underline{b}\alpha\cdots   &  \underline{b}B{\color{purple}x\beta}\cdots&
  \underline{b}\gamma\cdots \\
  \vdots &   \vdots &  \vdots\\
\end{array}\right]
-\left[\begin{array}{l|l|l}
  \underline{b}\alpha\cdots    &  \underline{b}Bx\beta'\cdots &
  \underline{b}\delta\cdots \\
  \underline{x}\beta'\cdots &
  \underline{x}\beta\cdots &
  \underline{x}\beta'\cdots \\
   \vdots   &  \vdots  &  \vdots\\
\end{array}\right]$$

\noindent
and an analogous argument to Cases I and II will show that there exists an additional row in $m_2$ where the pattern $x\beta'$ appears at most once and either we require two quadratic binomial moves involving all three rows or the third row can be used to apply a single move to reduce disagreements.

\emph{Mixed case.} Consider a polynomial $f\in\mathcal{I}^{\leq 3}(G)$ with both variables $s$ and $t$. In the standard form tableau, 2-treks come before 3-treks, hence the top row of $f$ will correspond to a 2-trek. If the disagreement can be reduced with another 2-trek we are in the already known setting of \cite{sullivant2008}, hence let us assume it requires 3-trek.

Let us start with the case where the first disagreement comes right after the top. We can then assume that 

$$f=\left[\begin{array}{l|l|l}
  \underline{a} \alpha\cdots   &  \underline{a}\beta\cdots  & \\
  \vdots &   \vdots &  \vdots\\
\end{array}\right]
-\left[\begin{array}{l|l|l}
  \underline{a}\gamma\cdots    &  \underline{a}\delta\cdots & \\
  \underline{a}\alpha\cdots & \underline{a}\varepsilon\cdots & \underline{a}\eta\cdots \\
   \vdots   &  \vdots  &  \vdots\\
\end{array}\right],$$

\noindent 
with $\alpha<\beta$ and $\gamma<\delta$. We assume $\alpha<\gamma$ but the case $\alpha=\gamma$ follows analogously. Note that we cannot assume anymore that the first columns of the first row in both terms are equal because the first path $a\alpha$ could be part of a 3-trek. We can apply the quadratic move

$$q=\left[\begin{array}{l|l|l}
 \underline{a}\gamma\cdots    &  \underline{a}\delta\cdots & \\
  \underline{a}\alpha\cdots & \underline{a}\varepsilon\cdots & \underline{a}\eta\cdots \\
\end{array}\right]-
\left[\begin{array}{l|l|l}
 \underline{a}\alpha\cdots    &  \underline{a}\delta\cdots & \\
  \underline{a}\gamma\cdots & \underline{a}\varepsilon\cdots & \underline{a}\eta\cdots \\
\end{array}\right]$$
in order to reduce disagreements provided $\gamma,\varepsilon,\eta$ are not all equal. If $\gamma=\varepsilon=\eta$, then there exists an additional third row in $m_2$

$$f=\left[\begin{array}{l|l|l}
  \underline{a} \alpha\cdots   &  \underline{a}\beta\cdots  & \\
  \vdots &   \vdots &  \vdots\\
\end{array}\right]
-\left[\begin{array}{l|l|l}
  \underline{a}\gamma\cdots    &  \underline{a}\delta\cdots & \\
  \underline{a}\alpha\cdots & \underline{a}\gamma\cdots & \underline{a}\gamma\cdots \\
   \vdots   &  \vdots  &  \vdots\\
\end{array}\right].$$

\noindent
with top $a$ that either corresponds to a 2-trek with no occurrences of $a\gamma$ or to a 3-trek with at most one occurence. Indeed, let $1+r+s$ be the number of rows with top $a$, where $s$ stands for the number of 3-treks. The maximum amount of $a\delta$ patterns in $m_2$ is $1+r+2s$ but in $m_1$ is at most $r+2s$, hence we cannot reach the maximum in $m_2$.
Note that the only situation where $a\gamma$ can occur in the first row of $m_1$ is if $\beta=\gamma$, but in this scenario we would be reducing the disagreement between $a\alpha$ and $a\delta$ instead.

From here we argue in the usual way: either 2 quadratic moves reduce disagreements, or the third row can be used directly to reduce. An analogous reasoning applies to disagreements that appear later on. \qedhere
\end{proof}

\bibliographystyle{plain}
\bibliography{refs}

\begin{thebibliography}{10}

\bibitem{comon:2008}
Pierre Comon, Gene Golub, Lek-Heng Lim, and Bernard Mourrain.
\newblock Symmetric tensors and symmetric tensor rank.
\newblock {\em SIAM Journal on Matrix Analysis and Applications},
  30(3):1254--1279, 2008.

\bibitem{MR3044565}
Jan Draisma, Seth Sullivant, and Kelli Talaska.
\newblock Positivity for {G}aussian graphical models.
\newblock {\em Adv. in Appl. Math.}, 50(5):661--674, 2013.

\bibitem{drton:2018}
Mathias Drton.
\newblock Algebraic problems in structural equation modeling.
\newblock In {\em The 50th anniversary of {G}r\"{o}bner bases}, volume~77 of
  {\em Adv. Stud. Pure Math.}, pages 35--86. Math. Soc. Japan, Tokyo, 2018.

\bibitem{drton:robeva:weihs:2020}
Mathias Drton, Elina Robeva, and Luca Weihs.
\newblock Nested covariance determinants and restricted trek separation in
  {G}aussian graphical models.
\newblock {\em Bernoulli}, 26(4):2503--2540, 2020.

\bibitem{haase2021existence}
C.~Haase.
\newblock {\em Existence of Unimodular Triangulations--Positive Results}.
\newblock Memoirs of the American Mathematical Society. American Mathematical
  Society, 2021.

\bibitem{handbook:chap3}
Thomas Kahle, Johannes Rauh, and Seth Sullivant.
\newblock Algebraic aspects of conditional independence and graphical models.
\newblock In {\em Handbook of graphical models}, Chapman \& Hall/CRC Handb.
  Mod. Stat. Methods, pages 61--80. CRC Press, Boca Raton, FL, 2019.

\bibitem{handbook}
Marloes Maathuis, Mathias Drton, Steffen Lauritzen, and Martin Wainwright,
  editors.
\newblock {\em Handbook of graphical models}.
\newblock Chapman \& Hall/CRC Handbooks of Modern Statistical Methods. CRC
  Press, Boca Raton, FL, 2019.

\bibitem{robeva2020}
Elina Robeva and Jean-Baptiste Seby.
\newblock Multi-trek separation in linear structural equation models.
\newblock {\em SIAM J. Appl. Algebra Geom.}, 5(2):278--303, 2021.

\bibitem{shimizu:lingam:2006}
Shohei Shimizu, Patrik~O. Hoyer, Aapo Hyv\"{a}rinen, and Antti Kerminen.
\newblock A linear non-{G}aussian acyclic model for causal discovery.
\newblock {\em J. Mach. Learn. Res.}, 7:2003--2030, 2006.

\bibitem{shimizu:directlingam:2011}
Shohei Shimizu, Takanori Inazumi, Yasuhiro Sogawa, Aapo Hyv\"{a}rinen,
  Yoshinobu Kawahara, Takashi Washio, Patrik~O. Hoyer, and Kenneth Bollen.
\newblock Direct{L}i{NGAM}: a direct method for learning a linear
  non-{G}aussian structural equation model.
\newblock {\em J. Mach. Learn. Res.}, 12:1225--1248, 2011.

\bibitem{handbook:chap:spirtes}
Peter Spirtes and Kun Zhang.
\newblock Search for causal models.
\newblock In {\em Handbook of graphical models}, Chapman \& Hall/CRC Handb.
  Mod. Stat. Methods, pages 439--469. CRC Press, Boca Raton, FL, 2019.

\bibitem{sturmfels1996grobner}
B.~Sturmfels.
\newblock {\em Grobner Bases and Convex Polytopes}.
\newblock Memoirs of the American Mathematical Society. American Mathematical
  Society, 1996.

\bibitem{sullivant2008}
Seth Sullivant.
\newblock Algebraic geometry of {G}aussian {B}ayesian networks.
\newblock {\em Advances in Applied Mathematics}, 40(4):482--513, 2008.

\bibitem{seth:book}
Seth Sullivant.
\newblock {\em Algebraic statistics}, volume 194 of {\em Graduate Studies in
  Mathematics}.
\newblock American Mathematical Society, Providence, RI, 2018.

\bibitem{MR2662356}
Seth Sullivant, Kelli Talaska, and Jan Draisma.
\newblock Trek separation for {G}aussian graphical models.
\newblock {\em Ann. Statist.}, 38(3):1665--1685, 2010.

\bibitem{VanOmmenMooij_UAI_17}
Thijs van Ommen and Joris~M. Mooij.
\newblock Algebraic equivalence of linear structural equation models.
\newblock In {\em Proceedings of the 33rd Annual Conference on {U}ncertainty in
  {A}rtificial {I}ntelligence ({UAI}-17)}, 2017.

\bibitem{wang:drton:2020}
Y.~Samuel Wang and Mathias Drton.
\newblock High-dimensional causal discovery under non-{G}aussianity.
\newblock {\em Biometrika}, 107(1):41--59, 2020.

\end{thebibliography}

\end{document}